\documentclass[10pt,reqno]{amsart} 

\usepackage[utf8]{inputenc}
\usepackage[a4paper]{geometry}

\usepackage{amsmath}
\usepackage{amsfonts}
\usepackage{amssymb}
\usepackage{amsthm}

\usepackage{natbib}
\usepackage{color}
\usepackage[dvipsnames]{xcolor}
\usepackage[colorlinks=true,linkcolor=blue,citecolor=blue,pdfborder={0 0 0}]{hyperref}

\usepackage{dsfont}
\usepackage{bm}

\usepackage{graphicx}
\usepackage{enumerate}
\usepackage{paralist}

\theoremstyle{plain} 
\newtheorem{theorem}{Theorem}[section]
\newtheorem{lemma}[theorem]{Lemma}
\newtheorem{proposition}[theorem]{Proposition}
\newtheorem{corollary}[theorem]{Corollary}

\theoremstyle{definition}
\newtheorem{definition}[theorem]{Definition} 
\newcommand{\cond}{\textbf{C}\!\!}

\theoremstyle{remark}
\newtheorem{remark}[theorem]{Remark}
\newtheorem*{notation*}{Notation}

\numberwithin{equation}{section}
\allowdisplaybreaks

\newcommand{\point}{\,\cdot\,}

\renewcommand{\le}{\leqslant}
\renewcommand{\leq}{\leqslant}
\renewcommand{\ge}{\geqslant}

\newcommand{\expec}{\mathbb{E}}
\newcommand{\law}{\mathcal{L}}
\newcommand{\GEV}{\operatorname{GEV}}

\newcommand{\GP}{\operatorname{GP}}
\newcommand{\GPR}{\GP_R}
\newcommand{\GPS}{\GP_S}
\newcommand{\GPT}{\GP_T}
\newcommand{\GPU}{\GP_U}

\newcommand{\abs}[1]{\lvert{#1}\rvert}

\newcommand{\dto}{\xrightarrow{d}}

\newcommand{\reals}{\mathbb{R}}
\newcommand{\eps}{\varepsilon}
\newcommand{\eqd}{\stackrel{d}{=}}
\newcommand{\diff}{\mathrm{d}}

\newcommand{\1}{\mathds{1}} 

\newcommand{\bzero}{\bm{0}}
\newcommand{\bone}{\bm{1}}
\newcommand{\binfty}{\bm{\infty}}

\newcommand{\bpi}{\bm{\pi}}
\newcommand{\balpha}{\bm{\alpha}}
\newcommand{\bEta}{\bm{\eta}}
\newcommand{\bgamma}{\bm{\gamma}}
\newcommand{\bmu}{\bm{\mu}}
\newcommand{\bsigma}{\bm{\sigma}}
\newcommand{\btau}{\bm{\tau}}

\newcommand{\bA}{\bm{A}}
\newcommand{\bR}{\bm{R}}
\newcommand{\bs}{\bm{s}}
\newcommand{\bS}{\bm{S}}
\newcommand{\bt}{\bm{t}}
\newcommand{\bT}{\bm{T}}
\newcommand{\bu}{\bm{u}}
\newcommand{\bU}{\bm{U}}
\newcommand{\bv}{\bm{v}}
\newcommand{\bV}{\bm{V}}
\newcommand{\bw}{\bm{w}}
\newcommand{\bW}{\bm{W}}
\newcommand{\bx}{\bm{x}}
\newcommand{\bX}{\bm{X}}
\newcommand{\by}{\bm{y}}
\newcommand{\bz}{\bm{z}}
\newcommand{\bZ}{\bm{Z}}

\newcommand{\mS}{{\max(\bS)}}
\newcommand{\mt}{{\max(\bt)}}
\newcommand{\mT}{{\max(\bT)}}
\newcommand{\mU}{{\max(\bU)}}
\newcommand{\mW}{{\max(\bW)}}

\newcommand{\Fbar}{\overline{F}}
\newcommand{\Hbar}{\overline{H}}

\newcommand{\Sj}{\mathcal{S}_j}
\newcommand{\Uj}{\mathcal{U}_j}

\definecolor{brown}{rgb}{0.59, 0.29, 0.0}
\definecolor{goldenrod}{rgb}{0.85, 0.65, 0.13}
\definecolor{gold(metallic)}{rgb}{0.83, 0.69, 0.22}
\definecolor{mygold}{rgb}{0.63, 0.49, 0.16}
\definecolor{asparagus}{rgb}{0.53, 0.66, 0.42}
\definecolor{olive}{rgb}{0.5, 0.5, 0.0}
\definecolor{antiquefuchsia}{rgb}{0.57, 0.36, 0.51}
\definecolor{golden(brown)}{rgb}{0.6, 0.4, 0.08} %
\definecolor{gray-asparagus}{rgb}{0.27, 0.35, 0.27}
\definecolor{glaucous}{rgb}{0.38, 0.51, 0.71}
\definecolor{airforceblue}{rgb}{0.36, 0.54, 0.66}
\definecolor{blue(munsell)}{rgb}{0.0, 0.5, 0.69}

\newcommand{\js}[1]{\textcolor{mygold}{\sffamily\small [JS: {#1}]}}
\newcommand{\changed}[1]{\textcolor{brown}{#1}}

\begin{document}


\title%
  [Multivariate generalized Pareto distributions]%
  {Multivariate generalized Pareto distributions: parametrizations, representations, and properties}
\author{Holger Rootzén
\and 
Johan Segers
\and
Jennifer L. Wadsworth
}

\address{Chalmers and Gothenburg University}
\email{hrootzen@chalmers.se}
\address{Universit\'e catholique de Louvain, Institut de statistique, biostatistique et sciences actuarielles, Voie du Roman Pays 20, B-1348 Louvain-la-Neuve, Belgium}
\email{johan.segers@uclouvain.be}
\address{Lancaster University}
\email{j.wadsworth@lancaster.ac.uk}
\date{\today}

\begin{abstract}
Multivariate generalized Pareto distributions arise as the limit distributions of exceedances over multivariate thresholds of random vectors in the domain of attraction of a max-stable distribution. These distributions can be parametrized and represented in a number of different ways. 
Moreover, generalized Pareto distributions enjoy a number of interesting stability properties. An overview of the main features of such distributions are given, expressed compactly in several parametrizations, giving the potential user of these distributions a convenient catalogue of ways to handle and work with generalized Pareto distributions.

\end{abstract}

\maketitle

\section{Introduction}

A core theme in univariate extreme-value analysis is to fit a generalized Pareto (GP) distribution to a sample of excesses over a high threshold. Since univariate GP distributions can be described in terms of a scale parameter and a shape parameter, statistical inference using frequentist or Bayesian likelihood techniques is straightforward, at least for values of the shape parameter at which the Fisher information matrix is finite.

For multivariate extremes, matters are more complicated. First, there is no universal definition of an exceedance of a multivariate threshold. Second, whatever the definition that is selected, the family of distributions proposed by asymptotic theory is no longer parametric.

Following \cite{rootzen2006}, we say that a sample point $\by \in \reals^d$ exceeds a multivariate threshold $\bu \in \reals^d$ as soon as one of its coordinates exceeds the corresponding threshold coordinate, i.e., $y_j > u_j$ for at least one $j = 1, \ldots, d$. In dimension $d = 2$, the shape of the excess region $\{ \by \in \reals^d : \by \not\leq \bu \}$ is that of the letter L upside-down; here and in what follows, inequalities between vectors are meant componentwise. The excess region covers a larger part of the sample space than the one for most other threshold exceedance definitions, for instance, that $\by$ exceeds $\bu$ when $\by > \bu$, i.e., $y_j > u_j$ for all $j = 1, \ldots, d$. 

The class of GP distributions that arises from the first definition of a multivariate exceedance is derived directly from the family of multivariate generalized extreme-valued (GEV) or max-stable distributions: \citet[Section~8.3]{beirlant2004} or \citet{rootzen2006}. Still, such multivariate GP distributions have enjoyed much less popularity than their univariate counterparts. One reason may be that the multivariate versions are mathematically more involved. Their support is (a subset of) $\{ \bx : \bx \not\leq \bzero \}$, the complement of the negative orthant, where $\bx = \by - \bu$ represents the excess vector, at least one coordinate of which is positive by definition. The unusual shape of the support introduces a nontrivial dependence structure uncommon to other families of multivariate distributions.

Our aim is to facilitate manipulation of multivariate GP distributions for the analysis of multivariate extremes. \cite{rootzen+s+w:2017} revisited multivariate GP distributions with an eye towards modelling. To facilitate the incorporation of physical constraints in the construction of GP models, these distributions were connected to a number of point process representations. In \cite{kiriliouk+rootzen+segers+wadsworth:2016}, the representations were used for the construction and calibration of parametric models admitting explicit density formulas.

To complete the picture, we focus here on a number of analytic properties of multivariate GP distributions. Our view is that a GP distribution is derived from a max-stable distribution from which it inherits the marginal parameters and the dependence structure after a suitable transformation. This construction directly motivates a number of stochastic representations of GP random vectors. Moreover, it leads to compact expressions and direct proofs of some interesting properties of multivariate GP distributions.

After recalling some basic definitions and properties in Section~\ref{sec:basics}, we introduce a number of parametrizations and stochastic representations in Sections~\ref{sec:param} and~\ref{sec:repr}, respectively. These results then provide the background against which we present compact formulas for probability densities (Section~\ref{sec:pdf}), marginal distributions (Section~\ref{sec:margin}), and copula-related objects (Section~\ref{sec:cop}). Finally, the family of multivariate GP distributions is stable with respect to thresholding (Section~\ref{sec:stable}) and, provided the margins have equal shape parameters, to certain linear transformations (Section~\ref{sec:lincomb}). All proofs are deferred to Appendices~\ref{app:proofs} and~\ref{app:suppl}.

\subsection*{Notation.} Throughout, the expressions $(1 + \gamma x)^{1/\gamma}$, $\log(1+\gamma x)/\gamma$, and $(x^\gamma - 1)/\gamma$ are to be read as their limits $\exp(x)$, $x$, and $\log(x)$, respectively, if $\gamma = 0$. When applied to vectors, mathematical operations such as addition, multiplication and exponentiation are to be interpreted componentwise, where scalars are recycled if necessary; for instance, for $\bgamma, \bx \in \reals^d$, we write $(1 + \bgamma \bx)^{1/\bgamma}$ for the vector $((1 + \gamma_j x_j)^{1/\gamma_j})_{j=1}^d$, with the earlier mentioned convention for $\gamma_j = 0$ applied to each component. We let $a \wedge b$ and $a \vee b$ denote $\min(a, b)$ and $\max(a, b)$, respectively, whereas for vectors, the minimum and the maximum are taken componentwise. Order relations between vectors are to be interpreted componentwise too. We write $\law( \xi )$ for the law of the random variable or vector $\xi$ and we let $\dto$ denote convergence in distribution. Bold face symbols denote vectors, usually of length $d$. Likewise, $\bzero = (0, \ldots, 0)$ and $\bone = (1, \ldots, 1)$, and $\binfty = (\infty, \ldots, \infty)$. For a vector $\bx$, we write $\max(\bx) = \max(x_1, \ldots, x_d)$. The indicator variable of the set $A$ is denoted by $\1(A)$.

\section{Basics}
\label{sec:basics}

Let $\bX$ be a $d$-variate random vector with cumulative distribution function (cdf) $F$. Suppose that there exist sequences of vectors $\bm{a}_n \in (0, \infty)^d$ and $\bm{b}_n \in \reals^d$ and a $d$-variate cdf $G$ with non-degenerate margins such that
\begin{equation}
\label{eq:FnG}
  F^n( \bm{a}_n \bx + \bm{b}_n ) \dto G(\bx), \qquad n \to \infty.
\end{equation}
The weak limit $G$ in \eqref{eq:FnG} is a $d$-variate max-stable or generalized extreme-value (GEV) distribution. The margins, $G_1, \ldots, G_d$, of $G$ are continuous, see \eqref{eq:Gj} below, so that the convergence in \eqref{eq:FnG} takes place for every $\bx \in \reals^d$. In particular, \eqref{eq:FnG} implies that, for all $\bx \in \reals^d$ such that $G(\bx) > 0$, we have
\begin{equation}
\label{eq:FnG:tail}
  \lim_{n \to \infty} n \{ 1 - F( \bm{a}_n \bx + \bm{b}_n ) \} = - \log G(\bx).
\end{equation}
We refer to \citet[Chapter~8]{beirlant2004} or \citet[Chapter~6]{dehaanferreira2006} for background on multivariate GEV distributions and their domains of attraction.

By an appropriate choice of the sequences $\bm{a}_n$ and $\bm{b}_n$, we can always ensure that
\begin{equation}
\label{eq:Gj0}
  0 < G_j(0) < 1, \qquad j = 1, \ldots, d. 
\end{equation}
Multivariate GEV distributions being positive quadrant dependent \citep{marshall1983}, we then have $0 < G( \bm{0} ) < 1$. By \eqref{eq:FnG:tail} and some elementary calculations, we find that, for all $\bm{x} \in \reals^d$ such that $G_j(x_j) > 0$ for all $j = 1, \ldots, d$,
\begin{equation}
\label{eq:cdflim}
  \lim_{n \to \infty}
  \Pr \{
    \bm{a}_n^{-1} (\bX - \bm{b}_n) \le \bx \mid \bX \not\leq \bm{b}_n
  \}
  =
  \frac{\log G( \bm{x} \wedge \bm{0} ) - \log G( \bm{x} )}{\log G(\bm{0})}.
\end{equation}

Let $\bEta \in [-\infty, 0)^d$ denote the vector of lower endpoints of the marginal distributions $G_1, \ldots, G_d$. From~\eqref{eq:cdflim}, it follows that
\begin{equation}
\label{eq:X2H}
  \law \{
    \bm{a}_n^{-1} (\bX - \bm{b}_n) \vee \bEta \mid \bX \not\leq \bm{b}_n
  \}
  \dto
  H,
  \qquad n \to \infty,
\end{equation}
where $H$ is the multivariate generalized Pareto (GP) distribution associated to $G$; notation $H = \GP(G)$. The support of $H$ is included in (but not necessarily equal to) the set $[\bEta, \binfty) \setminus [\bEta, \bzero]$, the set of all $\bx$ such that $x_j \ge \eta_j$ for all $j$ and $x_j > 0$ for at least one $j$. The function $H$ is determined by
\begin{equation}
\label{eq:G2H}
  H( \bm{x} ) 
  = 
  \frac{\log G( \bm{x} \wedge \bm{0} ) - \log G( \bm{x} )}{\log G(\bm{0})},
  \qquad \bm{x} \in (\bEta, \binfty).
\end{equation}
If $x_j = \eta_j$ for some $j = 1, \ldots, d$, then the value of $H( \bm{x} )$ is determined by continuity from the right. Note that $H$ may assign positive mass to the lower boundaries $\{ \bm{x} : x_j = \eta_j \}$, even if $\eta_j = -\infty$; see Proposition~\ref{prop:etaj} below.

The margins of $G$ are three-parameter generalized extreme-value distributions:
\begin{equation}
\label{eq:Gj}
  G_j(x_j) 
  = 
  \begin{cases}
    \exp [ - \{ 1 + \gamma_j (x_j - \mu_j) / \alpha_j \}^{-1/\gamma_j} ] 
      & \text{if $\gamma_j \ne 0$,} \\[1ex]
    \exp [ - \exp \{ - (x_j - \mu_j) / \alpha_j \} ] 
      & \text{if $\gamma_j = 0$,}
  \end{cases}
\end{equation}
for $j = 1, \ldots, d$ and $x_j \in \reals$ such that $\alpha_j + \gamma_j (x_j - \mu_j) > 0$; the parameter range is $\gamma_j \in \reals$, $\mu_j \in \reals$, and $\alpha_j \in (0, \infty)$. The dependence structure (copula) of $G$ can be described in many ways. In this paper we opt for the description in terms of the stable tail dependence function (stdf) $\ell : [0, \infty)^d \to [0, \infty)$ \citep{drees+h:98}. For $\bm{x} \in \reals^d$ such that $G_j(x_j) > 0$ for all $j = 1, \ldots, d$, we have
\begin{equation}
\label{eq:ell2G}
  G( \bm{x} ) = \exp [ - \ell \{ - \log G_1(x_1), \ldots, - \log G_d(x_d) \} ].
\end{equation}
The distribution $G$ is thus determined by the parameter vectors $\bgamma$, $\bmu$, and $\balpha$ together with the stdf $\ell$; notation $G = \GEV( \bmu, \bgamma, \balpha, \ell )$.

For later use we mention the fact that $\ell$ necessarily satisfies the following properties \citep{drees+h:98, ressel:2013}:
\begin{equation}
\label{eq:ell}
  \left.
  \begin{array}{@{\bullet\ }l}
    \text{$\ell$ is convex;} \\[1ex]
    \text{$\max(y_1, \ldots, y_d) \le \ell( \by ) \le y_1+\cdots+y_d$ for all $\by \in [0, \infty)^d$;} \\[1ex]
    \text{$\ell(c \by) = c \, \ell(\by)$ for all $(c, \by) \in [0, \infty) \times [0, \infty)^d$.}
  \end{array}
  \right\}
\end{equation}
A useful fact is also that a function $\ell : [0, \infty)^d \to [0, \infty)$ is a stdf if and only if there exists a random vector $\bV$ with values in $[0, \infty)^d$ and with $\expec[ V_j ] = 1$ such that
\begin{equation}
\label{eq:V2ell}
  \ell( \by ) = \expec[ \max( \by \bV ) ],
  \qquad \by \in [0, \infty)^d.
\end{equation}
Formula~\eqref{eq:V2ell} represents $\ell$ as (the restriction to $[0, \infty)^d$ of) a $D$-norm \citep{falk2010}. 
For a given $\bV$, the function $\ell$ in \eqref{eq:V2ell} is a stdf \citep[Lemma~3.1]{segers2012}. Given a stdf $\ell$, a possible choice for $\bV$ in \eqref{eq:V2ell} is $\bV = d \bW$, where $\bW$ is a random vector on the unit simplex $\Delta_{d-1} = \{ \bw \in [0, 1]^d : w_1 + \cdots + w_d = 1 \}$ whose law is proportional to the angular measure on $\Delta_{d-1}$ of the associated GEV distribution: indeed, we have $\ell( \by ) = d \int_{\Delta_{d-1}} \max( \by \bw ) \, \Pr(\bW \in \diff \bw)$ \citep[Theorem~6.1.14]{pickands81,dehaanferreira2006}. The random vector $\bV$ generating $\ell$ is not unique in distribution. Specific constructions will be considered in Section~\ref{sec:repr}.

\section{Parametrizations}
\label{sec:param}

If $H$ is determined by $G$ and if $G$ is determined by $( \bmu, \bgamma, \balpha, \ell )$, then so is $H$. However, this is not a convenient way to parametrize $H$, because the parameter vectors $\bmu$ and $\balpha$ are not identifiable from $H$. If $G = \GEV( \bmu, \bgamma, \balpha, \ell )$, then $G^t \sim \GEV( \bmu(t), \bgamma, \balpha(t), \ell )$ for all $t \in (0, \infty)$, where 
\begin{align} 
\label{eq:mut_alphat}
  \bmu(t) 
  &= \bmu + \balpha (t^{\bgamma} - 1) / \bgamma, & 
  \balpha(t) 
  &= t^{\bgamma} \balpha.
\end{align}
Still, if $H = \GP( G )$, then also $H = \GP( G^t )$ for all $t \in (0, \infty)$. All GEV distributions $G^t$ thus generate the same GP distribution $H$. The GP distribution describes the distribution of sample points given that they exceed a high threshold, but not the exceedance probability itself. This explains the loss of one parameter with respect to a full point process model, which has the same number of parameters as the GEV model. The phenomenon already occurs in the univariate case, where GP and GEV distributions have two and three parameters, respectively. Another way to understand the difference in the number of parameters is through the fact that the vector of componentwise maxima of a Poisson number of independent random vectors with common GP distribution $H$ has a distribution function which in its upper tail is equal to the GEV distribution $G^t$, where $t$ is proportional to the expectation of the Poisson random variable. 

The lack of identifiability of some of the GEV parameters from the associated GP distribution is one reason why we look for other parametrizations for $H$. Another reason for doing so is that GP distributions enjoy a number of interesting properties and representations, and some of these are more clearly understood and expressed in other parametrizations.

Let $G = \GEV( \bmu, \bgamma, \balpha, \ell )$ and write
\[
  \bsigma = \balpha - \bgamma \bmu.
\]
The requirement that $0 < G_j(0) < 1$ is equivalent to the requirement that $\sigma_j > 0$. To ensure \eqref{eq:Gj0}, we will therefore assume that $\bsigma \in ( \bzero, \binfty ) = (0, \infty)^d$. Recall $\bmu(t)$ and $\balpha(t)$ in \eqref{eq:mut_alphat} and note that $\balpha(t) - \bgamma \bmu(t) = \bsigma$ for all $t \in (0, \infty)$, that is, $\bsigma$ is a common parameter for all GEV distributions $G^t$.

\begin{proposition}
\label{prop:G2H}
Let $G$ be $\GEV( \bmu, \bgamma, \balpha, \ell )$ with $\bsigma = \balpha - \bgamma \bmu \in (\bzero, \binfty)$. Let $H$ be $\GP( G )$. For $\bx \in \reals$ such that $\sigma_j + \gamma_j x_j > 0$ for all $j$, we have
\begin{equation}
\label{eq:H:pi-ell}
  H( \bx )
  =
  \ell \{ \bpi \, ( 1 + \bgamma \tfrac{\bx \wedge \bzero}{\bsigma} )^{-1/\bgamma} \}
  -
  \ell \{ \bpi \, ( 1 + \bgamma \tfrac{\bx}{\bsigma} )^{-1/\bgamma} \}
\end{equation}
where, for $j = 1, \ldots, d$,
\begin{align*}
  \pi_j &= \tau_j / \ell(\btau) \in (0, 1], \\
  \tau_j &= -\log G_j(0) = (1 - \gamma_j \mu_j / \alpha_j)^{-1/\gamma_j} \in (0, \infty), \\
  \ell(\btau) &= - \log G( \bm{0} ).
\end{align*}
\end{proposition}

\begin{proposition}
\label{prop:H:identify}
In Proposition~\ref{prop:G2H}, each of $\bgamma$, $\bsigma$, $\bpi$, and $\ell$ are identifiable from $H$, and $\ell( \bpi ) = 1$. More precisely, writing $\Hbar = 1 - H$ and $\Hbar_j = 1 - H_j$ for $j = 1, \ldots, d$, we have, for $\bx \in [ \bzero, \binfty )$,
\begin{align}
\label{eq:Hj0}
  \overline{H}_j(0) 
  &= \pi_j, \\
\label{eq:Hj|0}
  \frac{\Hbar_j(x_j)}{\Hbar_j(0)}
  &= (1 + \gamma_j x_j / \sigma_j)^{-1/\gamma_j}, 
  \qquad \text{provided $\sigma_j + \gamma_j x_j > 0$}, \\
\label{eq:Hbar:ell}
  \Hbar( \bm{x} )
  &=
  \ell \{ \Hbar_1(x_1), \ldots, \Hbar_d(x_d) \}.
\end{align}
Furthermore, we have $\btau = \ell( \btau ) \, \bpi$, so that the vector $\btau$ is identifiable up to a constant multiple.
\end{proposition}

In view of Propositions~\ref{prop:G2H} and~\ref{prop:H:identify}, we express \eqref{eq:H:pi-ell} as
\begin{equation}
\label{eq:H:param:pi}
  H = \GP( \bsigma, \bgamma, \bpi, \ell ).
\end{equation}
This yields a parametrization of $H$ in terms of $\bgamma \in \reals^d$, $\bsigma \in (0, \infty)^d$, $\bpi \in (0, 1]^d$, and a stdf $\ell$ such that $\ell( \bpi ) = 1$. All four components of $( \bsigma, \bgamma, \bpi, \ell )$ are identifiable from $H$. However, the nonlinear constraint $\ell( \bpi ) = 1$ may perhaps be impractical when doing inference. Therefore, we also propose the alternative parametrization
\begin{equation}
\label{eq:H:param:tau}
  H = \GP( \bsigma, \bgamma, \btau, \ell ),
\end{equation}
where $\bpi$ in \eqref{eq:H:param:pi} has been replaced by a vector $\btau \in (0, \infty)^d$ which is identifiable only up to a positive multiplicative constant. This lack of identifiability can easily be remedied by adding a constraint such as $\tau_1 + \cdots + \tau_d = c$, where $c$ is a positive constant, for instance $c = 1$ or $c = d$. The vector $\bpi$ can be reconstructed from $\btau$ and $\ell$ via $\bpi = \btau / \ell(\btau)$. Since $\ell$ is homogeneous, multiplying $\btau$ by a positive constant does not affect $\bpi$. A valid choice for $\btau$ would be $\bpi$ itself, which justifies the use of the same notation in \eqref{eq:H:param:pi} and \eqref{eq:H:param:tau}.

\section{Stochastic representations}
\label{sec:repr}

In the parametrization $\bX \sim \GP( \bsigma, \bgamma, \bpi, \ell)$, the parameter vectors $\bsigma \in (\bzero, \binfty)$ and $\bgamma \in \reals^d$ represent marginal scale and shape vectors, respectively.

\begin{proposition}
\label{prop:standard}
We have $\bX \sim \GP( \bsigma, \bgamma, \bpi, \ell )$ if and only if
\begin{equation}
\label{eq:Z2X}
  \bX = \bsigma \, \frac{e^{\bgamma \bZ} - 1}{\bgamma},
  \qquad 
  \text{with}\ \bZ \sim \GP( \bone, \bzero, \bpi, \ell ).
\end{equation}
The support of $\bZ$ is contained in $[-\binfty, \binfty) \setminus [-\binfty, \bzero]$ and its cdf is given by
\begin{equation}
\label{eq:H:Z}
  \Pr[ \bZ \le \bz ]
  =
  \ell ( \bpi e^{-(\bz \wedge 0)} )
  -
  \ell ( \bpi e^{-\bz} ),
  \qquad 
  \bz \in \reals^d.
\end{equation}
\end{proposition}

In view of Proposition~\ref{prop:standard}, we can reduce the study of many aspects of general GP distributions to the special case of GP distributions with $\bsigma = \bone$ and $\bgamma = \bzero$.

\citet[Sections~4 and~5]{rootzen+s+w:2017} introduced a number of stochastic representations of (standardized) GP random vectors. The representations were derived from that of a multivariate GEV distribution as the law of the vector of componentwise maxima of the points of certain point processes. Here, we derive these representations from scratch via that of the stdf $\ell$ in \eqref{eq:V2ell}. We also connect the representations to the parametrization in terms of $\bpi$ and $\ell$.

\begin{definition}
\label{def:S}
A random vector $\bS$ taking values in $[-\infty, 0]^d$ is called a \emph{spectral random vector} if the following two conditions hold:
\begin{compactenum}[({S}1)]
\item $\Pr[ \max(S_1, \ldots, S_d) = 0 ] = 1$;
\item $\Pr[ S_j > -\infty ] > 0$ for all $j = 1, \ldots, d$.
\end{compactenum}
\end{definition}

\begin{theorem}
\label{thm:S:direct}
Let $\bS$ be a spectral random vector and let $E$ be a unit exponential random variable independent of $\bS$. Then
\[
  \bS + E \sim \GP( \bone, \bzero, \bpi, \ell )
\]
where $\bpi$ and $\ell$ are given by
\begin{align}
\label{eq:S2pi}
  \pi_j 
  &= \expec[ e^{S_j} ], \qquad j = 1, \ldots, d, \\
\label{eq:S2ell}
  \ell( \by ) 
  &= 
  \expec \left[
    \max( \by e^{\bS} / \bpi )
  \right],
  \qquad \by \in [0, \infty)^d.
\end{align}
The associated cdf is given by
\begin{equation*}
  H( \bz ) = 1 - \expec [ 1 \wedge e^{ \max( \bS - \bz ) } ].
\end{equation*}
\end{theorem}

\begin{theorem}
\label{thm:S:converse}
For every pair $(\bpi, \ell)$ where $\bpi \in (0, 1]^d$ and where $\ell$ is a stdf with $\ell( \bpi ) = 1$, there exists a spectral random vector $\bS$, unique in distribution, such that $\bZ \sim \GP( \bone, \bzero, \bpi, \ell )$ can be represented in distribution as
\begin{equation}
\label{eq:S2Z}
  \bZ \eqd \bS + E,
\end{equation}
with $E$ a unit exponential random variable, independent of $\bS$.
\end{theorem}

\begin{remark}[Link with other representations]
Working on the unit Fr\'echet scale, $\gamma_j = 1$ for all $j = 1, \ldots, d$, rather than on the Gumbel scale, the representation \eqref{eq:S2Z} states that $e^{\bZ} - 1$ is equal in distribution to $Y \bm{W} - 1$, where $Y = e^E$ is a unit Pareto random variable and where $\bm{W} = e^{\bm{S}}$ is a random vector concentrated on $\{ \bm{w} \in [0, \infty)^d : \max( \bm{w} ) = 1 \}$ and independent of $Y$. Up to the translation by $-1$, such distributions have already been considered in the literature before. They arise as special cases of Corollary~3.2 in \cite{basrak+s:2009} in the study of the possible weak limits of $\bm{X} / x$ given that $\lVert \bm{X} \rVert > x$ as $x \to \infty$, where $\bm{X}$ is a regularly varying random vector, and they also appear in \cite{ferreira2014} in the context of generalized Pareto processes. The random vector $(e^{S_j} / (d \pi_j))_{j=1}^d$ is supported on the unit simplex $\Delta_{d-1}$ and its distribution is proportional to the angular measure of the associated GEV distribution, as explained right after \eqref{eq:V2ell}.
\end{remark}

In view of Theorems~\ref{thm:S:direct} and~\ref{thm:S:converse}, there is a one-to-one relation between pairs $(\bpi, \ell)$ satisfying $\ell(\bpi) = 1$ on the one hand and distributions, $\nu$, of spectral random vectors $\bS$ on the other hand. We therefore write
\begin{equation*}
  H = \GPS( \bsigma, \bgamma, \nu )
\end{equation*}
for the GP distribution $H = \GP( \bsigma, \bgamma, \bpi, \ell )$ with $(\bpi, \ell)$ determined by $\nu = \law(\bS)$.

Combining \eqref{eq:Z2X} and \eqref{eq:S2Z}, we find that a general GP random vector $\bX \sim \GP( \bsigma, \bgamma, \bpi, \ell )$ can be represented as
\begin{equation}
\label{eq:S2X}
  \bX \eqd \bsigma \, \frac{e^{\bgamma( \bS + E )} - 1}{\bgamma}
\end{equation}
where $\bS$ is the spectral random vector associated to $(\bpi, \ell)$ and where $E$ is a unit exponential random variable independent of $\bS$. Representation~\eqref{eq:S2X} is convenient for model construction and Monte Carlo simulation provided we have a handle on the spectral random vector $\bS$.

The requirement that $\max( \bS ) = 0$ almost surely in Definition~\ref{def:S} may perhaps look difficult to ensure, but actually, it is not. We describe two constructions for doing so.

\begin{proposition}
\label{prop:S:I}
Let $\bT$ be a random vector taking values in $[-\infty, \infty)^d$ such that the following two conditions hold:
\begin{compactenum}[({T}1)]
\item $\Pr[ T_j > - \infty ] > 0$ for all $j = 1, \ldots, d$;
\item $\Pr[ \mT > - \infty ] = 1$.
\end{compactenum}
Then
\[
  \bS = \bT - \mT
\]
is a spectral random vector (Definition~\ref{def:S}) and the associated GP distribution, $\GP( \bone, \bzero, \bpi, \ell ) = \GP_S( \bzero, \bone, \law(\bS) )$, is determined by
\begin{align}
\label{eq:T2pi}
  \pi_j &= \expec[ e^{T_j - \mT} ],
  && j = 1, \ldots, d;
  \\
\label{eq:T2ell}
  \ell( \by ) 
  &= 
  \expec \left[ 
    \max_{j=1,\ldots,d} \left\{ 
      y_j \, \frac{e^{T_j - \mT}}{\expec[ e^{T_j - \mT} ]} 
    \right\}
  \right],
  && \by \in [0, \infty)^d.
\end{align}
The associated cdf is given by
\begin{equation}
\label{eq:T2H}
  H( \bz ) = 1 - \expec [ 1 \wedge e^{ \max( \bT - \bz ) - \mT } ].
\end{equation}
\end{proposition}

\begin{proposition}
\label{prop:S:II}
Let $\bU$ be a random vector taking values in $[-\infty, \infty)^d$ such that the following condition holds:
\begin{compactitem}
\item[(U)]
$0 < \expec[ e^{U_j} ] < \infty$ for all $j = 1, \ldots, d$.
\end{compactitem}
Let $\bS$ be defined in distribution by
\begin{equation}
\label{eq:U2S}
  \Pr[ \bS \in \point ] 
  =
  \frac%
    {\expec \left[ \1 \{ \bU - \mU \in \point \} \, e^{\mU} \right]}%
    {\expec [ e^{\mU} ]}.
\end{equation}
Then $\bS$ is a spectral random vector as in Definition~\ref{def:S} and the associated GP distribution, $\GP( \bone, \bzero, \bpi, \ell ) = \GP_S( \bzero, \bone, \law(\bS) )$, is determined by
\begin{align}
\label{eq:U2pi}
  \pi_j &= \frac{\expec[ e^{U_j} ]}{\expec[ e^{\mU} ]}, 
  && j = 1, \ldots, d;
  \\
\label{eq:U2ell}
  \ell( \by ) 
  &= 
  \expec \left[ 
    \max_{j=1,\ldots,d} \left\{ 
      y_j \, \frac{e^{U_j}}{\expec[ e^{U_j} ]} 
    \right\}
  \right],
  && \by \in [0, \infty)^d.
\end{align}
The associated cdf is given by
\begin{equation}
\label{eq:U2H}
  H( \bz ) = 1 - \frac%
    {\expec [ e^{\mU} \vee e^{\max( \bU - \bz )} ]}%
    {\expec [ e^{\mU} ]}.
\end{equation}
\end{proposition}

If $\nu$ is the law of the random vector $\bT$ in Proposition~\ref{prop:S:I}, we write
\begin{equation}
\label{eq:GPT}
  \GPT( \bsigma, \bgamma, \nu )
  =
  \GP( \bsigma, \bgamma, \bpi, \ell )
\end{equation}
with $(\bpi, \ell)$ determined by $\bT$ as in \eqref{eq:T2pi}--\eqref{eq:T2ell}. Similarly, if $\nu$ is the law of the random vector $\bU$ in Proposition~\ref{prop:S:II}, we write
\begin{equation}
\label{eq:GPU}
  \GPU( \bsigma, \bgamma, \nu )
  =
  \GP( \bsigma, \bgamma, \bpi, \ell )
\end{equation}
with $(\bpi, \ell)$ determined by $\bU$ as in \eqref{eq:U2pi}--\eqref{eq:U2ell}.

Practical modelling considerations led us in \cite{rootzen+s+w:2017} to consider multivariate GP distribution functions of the form
\begin{equation}
\label{eq:R2H}
  H_R( \bx )
  =
  \frac%
    {
      \int_0^\infty
	\{
	  F_{\bR}( t^{\bgamma}( \bx + \bsigma / \bgamma ) )
	  -
	  F_{\bR}( t^{\bgamma}( (\bx \wedge \bzero) + \bsigma / \bgamma ) )
	\} \,
      \diff t
    }%
    {%
      \int_0^\infty 
	\Fbar_{\bR}( t^{\bgamma} \bsigma / \bgamma ) \, 
      \diff t
    }.
\end{equation}
a distribution denoted as $\GPR(\bsigma, \bgamma, \law(\bR))$.
The marginal parameter vectors are $\bsigma, \bgamma \in (\bzero, \binfty)$, while $\bR$ is a random vector on $[\bzero, \binfty)$ such that $0 < \expec[ R_j^{1/\gamma_j} ] < \infty$ for all $j = 1, \ldots, d$. (The cases where $\gamma_j = 0$ or $\gamma_j < 0$ for some or all $j$ were considered in \cite{rootzen+s+w:2017} as well.) Further, $F_{\bR}$ is the cdf of $\bR$ and $\Fbar_{\bR} = 1 - F_{\bR}$, while the argument, $\bx$, in \eqref{eq:R2H} is such that $\sigma_j + \gamma_j x_j > 0$ for all $j = 1, \ldots, d$.

\begin{proposition}
\label{prop:R}
For $\bsigma, \bgamma \in (\bzero, \binfty)$, the function $H_R$ in \eqref{eq:R2H} is the cdf of the $\GPU( \bsigma, \bgamma, \law(\bU) )$ distribution, where $\bU = \bgamma^{-1} \log( \bgamma \bR / \bsigma )$.
\end{proposition}

\begin{remark}[Identifying $\nu$]
By the uniqueness statement in Theorem~\ref{thm:S:converse}, the law, $\nu$, of $\bm{S}$ in the $\GPS$ representation can be identified from the GP distribution. This is not the case for the $\GPT$ or $\GPU$, representations, however: adding a common, independent random variable $\xi$ with $\expec[ e^\xi ] < \infty$ to all components $T_j$ or $U_j$ in Propositions~\ref{prop:S:I} or \ref{prop:S:II} does not change the resulting spectral random vector $\bS$ nor the GP parameter $(\bpi, \ell)$. Still, if the laws of $\bT$ or $\bU$ are known up to some finite-dimensional parameter, then this parameter may well be identifiable from the associated GP distribution. This is to be investigated on a case-by-case basis, for instance by inspection of the density functions (Section~\ref{sec:pdf}).
\end{remark}

\begin{remark}[The role of $\nu$]
If $\nu$ is already the law of a spectral random vector $\bS$, then the transformations in Propositions~\ref{prop:S:I} and~\ref{prop:S:II} leave $\nu$ invariant, so that $\GPS( \bsigma, \bgamma, \nu )$, $\GPT( \bsigma, \bgamma, \nu )$ and $\GPU( \bsigma, \bgamma, \nu )$ are all the same. In that sense, the $\GPS$ notation is redundant. Still, we keep using it because of the uniqueness of $\law(\bS)$ and because of its role as common starting point for the $\GPT$ and $\GPU$ constructions.

If $\nu$ is not the distribution of a spectral random vector, however, then $\GPS( \bsigma, \bgamma, \nu )$ is not defined, whereas $\GPT( \bsigma, \bgamma, \nu)$ and $\GPU( \bsigma, \bgamma, \nu )$ are different in general.
\end{remark}

\begin{remark}[Model construction]
The probability measure $\nu$ in \eqref{eq:GPT} and \eqref{eq:GPU} need only satisfy conditions (T1)--(T2) or (U) in Propositions~\ref{prop:S:I} or~\ref{prop:S:II}, respectively. This makes the $\GPT$ and $\GPU$ representations attractive for model construction. Some examples include the multivariate normal distribution or distributions with independent components. More involved models arise when $\nu$ is the joint distribution of a vector of partial sums. Specific constructions and case studies are worked out in \cite{kiriliouk+rootzen+segers+wadsworth:2016}. The possibilities are endless and constitute a potentially fruitful research avenue.
\end{remark}

\begin{remark}[$D$-norms]
Formulas \eqref{eq:S2ell}, \eqref{eq:T2ell} and \eqref{eq:U2ell} represent the stdf $\ell$ as a $D$-norm as in \eqref{eq:V2ell}. In \cite{falk2010}, $D$-norms are linked to multivariate GP distributions as well.
\end{remark}


\section{Densities}
\label{sec:pdf}

For statistical inference, it is highly useful to know the probability density functions of the GP distributions constructed via the methods in Section~\ref{sec:repr}. Most of the results of this section can also be found in \cite{rootzen+s+w:2017}, but for completeness, we give their proofs in Appendix~\ref{app:suppl}. Theorem~\ref{thm:pdf:S} is new.

By Proposition~\ref{prop:standard}, it is sufficient to consider the standardized case  $\bsigma = \bone$ and $\bgamma = \bzero$. The density in case of general $(\bsigma, \bgamma)$ can then be found by the componentwise increasing transformation $\bz \mapsto \bx = \bsigma (e^{\bgamma \bz} - 1) / \bgamma$. 

\begin{lemma}
\label{lem:Z2X:pdf}
If $\bZ \sim \GP( \bone, \bzero, \bpi, \ell )$ is absolutely continuous with Lebesgue density $h_{\bZ}$, then $\bX \sim \GP( \bsigma, \bgamma, \bpi, \ell )$ in \eqref{eq:Z2X} is absolutely continuous with Lebesgue density $h_{\bX}$ given by
\begin{equation}
\label{eq:Z2X:pdf}
  h_{\bX}( \bx ) 
  =
  h_{\bZ} \left( \bgamma^{-1} \log( 1 + \bgamma \bx / \bsigma ) \right) \,
  \prod_{j=1}^d \frac{1}{\sigma_j + \gamma_j x_j},
\end{equation}
for $\bx \in \reals^d$ such that $\bx \not\le \bzero$ and $\sigma_j + \bgamma_j x_j > 0$ for all $j = 1, \ldots, d$.
\end{lemma}

Next we give expressions for the density of $\bZ = \bS + E$ in Theorem~\ref{thm:S:direct} for the stochastic representations of $\bS$ via $\bT$ and $\bU$ in Proposition~\ref{prop:S:I} and~\ref{prop:S:II}, respectively.

\begin{theorem}
\label{thm:pdf:I}
If $\nu = \law(\bT)$, with $\bT$ as in Proposition~\ref{prop:S:I}, has support included in $\reals^d$ and is absolutely continuous with Lebesgue density $f_{\bT}$, then the $\GPT( \bone, \bzero, \nu )$ distribution has Lebesgue density $h$ given by
\begin{equation}
\label{eq:T2h}
  h( \bz ) 
  = 
  \1( \bz \not\le \bzero ) \, \frac{1}{e^{\max(\bz)}} 
  \int_0^\infty f_{\bT}( \bz + \log t ) \, t^{-1} \, \diff t.
\end{equation}
\end{theorem}

\begin{theorem}
\label{thm:pdf:II}
If $\nu = \law(\bU)$, with $\bU$ as in Proposition~\ref{prop:S:II}, has support included in $\reals^d$ and is absolutely continuous with Lebesgue density $f_{\bU}$, then the $\GPU( \bone, \bzero, \nu )$ distribution has Lebesgue density $h$ given by
\begin{equation}
\label{eq:U2h}
  h( \bz ) 
  = 
  \1( \bz \not\le \bzero ) \, \frac{1}{\expec[ e^{\mU} ]} \,
  \int_0^\infty f_{\bU}( \bz + \log t ) \, \diff t.
\end{equation}
\end{theorem}

Combining Propositions~\ref{prop:S:II} and~\ref{prop:R}, we also find the density of $H_R$ in \eqref{eq:R2H} in terms of the one of $\bR$.

\begin{corollary}
\label{cor:R:pdf}
Let $\bsigma, \bgamma \in (\bzero, \binfty)$ and let $\bR$ be a random vector with values in $(\bzero, \binfty)$ such that $\expec[ R_j^{1/\gamma_j} ] < \infty$ for all $j=1,\ldots,d$. If $\bR$ is absolutely continuous with Lebesgue density $f_{\bR}$, then the GP cdf $H_R$ in \eqref{eq:R2H} is absolutely continuous with Lebesgue density
\[
  h( \bx ) 
  = 
  \1( \bx \not\le \bzero )
  \frac{1}{ \expec[ \max \{ (\bgamma \bR / \bsigma)^{1/\bgamma} \} ] }
  \int_0^\infty
    f_{\bR} \bigl( t^{\bgamma} ( \bx + \bsigma / \bgamma ) \bigr) \,
    t^{\sum_{j=1}^d \gamma_j} \,
  \diff t
\]
for $\bx$ such that $\sigma_j + \gamma_j x_j > 0$ for all $j = 1, \ldots, d$.
\end{corollary}

By definition, a spectral random vector $\bS$ is supported on the Lebesgue null set $\{ \bm{s} : \max(\bm{s}) = 0 \}$. Still, it may be absolutely continuous with respect to the $(d-1)$-dimensional Lebesgue measure on that set with some density function $f_{\bS}$, and then the associated $\GPS$ distribution is absolutely continuous with respect to the $d$-dimensional Lebesgue measure.

\begin{theorem}
\label{thm:pdf:S}
Let $\bS$ be a spectral random vector with Lebesgue density $f_{\bS}$ defined on $\{\bs \in \mathbb{R}^d: \max(\bs)=0\}$. Let $\bZ=\bS+E$ be the associated GP random vector. Then the density of $\bZ$ is given by
\[
  h(\bz) = \1( \bz \not\le \bzero ) \, f_{\bS}(\bz-\max(\bz)) \, e^{-\max(\bz)}.
\]
\end{theorem}

\section{Margins and lower boundaries}
\label{sec:margin}

The margins of a multivariate GP distribution are in general not univariate GP distributions; indeed, their supports are not necessarily included in $[0, \infty)$. Still, by \eqref{eq:Hj|0}, their conditional versions are univariate GP.

\begin{proposition}
\label{prop:margins}
For $j = 1, \ldots, d$, the $j$-th marginal distribution function, $H_j$, of a GP distribution $H = \GP( \bsigma, \bgamma, \bpi, \ell ) = \GPS( \bsigma, \bgamma, \law(\bS) )$ with $(\sigma_j, \gamma_j) = (1, 0)$ is given as follows: for $z_j \in \reals$,
\begin{align}
\label{eq:Hj:z}
  H_j(z_j)
  &=
  \ell \bigl( 
    \pi_1, \ldots, \pi_{j-1}, \pi_j e^{-(z_j \wedge 0)}, \pi_{j+1}, \ldots, \pi_d 
  \bigr)
  -
  \pi_j e^{-z_j} \\
\label{eq:Hj:S}
  &=
  \expec[ \max\{ e^{S_1}, \ldots, e^{S_{j-1}}, e^{S_j - (z_j \wedge 0)}, e^{S_{j+1}}, \ldots, e^{S_d} \} ]
  -
  e^{-z_j} \, \expec[ e^{S_j} ].
\end{align}
\begin{itemize}
\item
If $z_j \ge 0$, the right-hand sides simplify to $1 - \pi_j e^{-z_j} = 1 - e^{-z_j} \expec[e^{S_j}]$.

\item
For $H = \GPT( \bsigma, \bgamma, \law(\bT) )$, replace $S_k$ by $T_k - \mT$ for all $k = 1, \ldots, d$.

\item
For $H = \GPU( \bsigma, \bgamma, \law(\bU) )$, replace $S_k$ by $U_k$ for all $k = 1, \ldots, d$ and divide everything by $\expec[ e^{\mU} ]$. 

\item
For general $(\sigma_j, \gamma_j) \in (0, \infty) \times \reals$, replace $z_j$ by $(1 + \gamma_j x_j / \sigma_j)^{-1/\gamma_j}$ for real $x_j$ such that $\sigma_j + \gamma_j x_j > 0$.
\end{itemize}
\end{proposition}

Recall that $\eta_j$ is the lower endpoint of $G_j$, the $j$-th margin of the GEV $G$ generating $H$:
\begin{equation}
\label{eq:eta_j}
  \eta_j =
  \begin{cases}
    - \sigma_j / \gamma_j & \text{if $\gamma_j > 0$,} \\[1ex]
    - \infty & \text{if $\gamma_j \le 0$.}
  \end{cases}
\end{equation}
[Note that this lower endpoint is common for all GEV distributions $G^t$ with $t \in (0, \infty)$.] Letting $z_j$ decrease to $-\infty$ in \eqref{eq:Hj:z} yields the probability mass assigned by $H = \GP( \bsigma, \bgamma, \bpi, \ell)$ to the hyperplane $\{ \bx : x_j = \eta_j \}$.

\begin{proposition}
\label{prop:etaj}
Let $H = \GP( \bsigma, \bgamma, \bpi, \ell )$, let $j = 1, \ldots, d$, and let $\eta_j$ be as in \eqref{eq:eta_j}. We have
\begin{align}
\label{eq:Hj:etaj:1}
  H_j( \{ \eta_j \} )
  &= 
  \lim_{y_j \to \infty}
  \{
    \ell \bigl( \pi_1, \ldots, \pi_{j-1}, \pi_j y_j, \pi_{j+1}, \ldots, \pi_d \bigr)
    - 
    \pi_j y_j
  \} \\
\label{eq:Hj:etaj:2}
  &=
  \lim_{\eps \downarrow 0}
  \eps^{-1}
  \{
    \ell \bigl( 
      \eps \pi_1, \ldots, \eps \pi_{j-1}, \pi_j, \eps \pi_{j+1}, \ldots, \eps \pi_d
    \bigr)
    -
    \pi_j
  \}.
\end{align}
If the first-order partial derivatives of $\ell$ exist and are continuous on a neighbourhood of the point $\pi_j \bm{e}_j = (0, \ldots, 0, \pi_j, 0, \ldots, 0)$, then also
\begin{equation}
\label{eq:Hj:etaj:3}
  H_j( \{ \eta_j \} )
  =
  \sum_{\substack{k=1,\ldots,d \\ k \ne j}}
  \pi_k \, \dot{\ell}_k(\pi_j \bm{e}_j). 
\end{equation}
Finally, in the $\GPS$, $\GPT$ and $\GPU$ representations, we have, respectively,
\begin{equation}
\label{eq:Hj:etaj:STU}
  H_j( \{ \eta_j \} )
  =
  \left\{
    \begin{array}{l}
      \Pr( S_j = -\infty ), \\[1ex]
      \Pr( T_j = -\infty ), \\[1ex]
      \expec[ e^{\mU} \, \1\{ U_j = -\infty \} ] \, / \, \expec[ e^{\mU} ].
    \end{array}
  \right.
\end{equation} 
\end{proposition}

Equation~\eqref{eq:Hj:z} implies that the margins of $H$ are continuous, except for a possible atom at the lower endpoints $\eta_1,\ldots,\eta_d$. Weak convergence of multivariate threshold exceedances in \eqref{eq:X2H} implies that, for all $j = 1, \ldots, d$ and all $x_j > \eta_j$, we have
\[
  \lim_{n \to \infty}
  \Pr\{ a_{n,j}^{-1} ( X_{n,j} - b_{n,j} ) \le x_j \mid \bX \not\leq \bm{b}_n \}
  =
  H_j( x_j ).
\]
Taking the limit as $x_j \downarrow \eta_j$, we obtain a statistical interpretation of the probability mass (if any) assigned by $H$ to $\{ \bx : x_j = \eta_j \}$:
\[
  H_j( \{ \eta_j \} )
  =
  \lim_{x_j \downarrow \eta_j}
  \lim_{n \to \infty}
  \Pr\{ a_{n,j}^{-1} ( X_{n,j} - b_{n,j} ) \le x_j \mid \bX \not\leq \bm{b}_n \}.  
\]
In words, $H_j( \{ \eta_j \} )$ represents the probability that the rescaled (negative) excess in the $j$-th component, conditionally on a (positive) excess in some of the other components, drops below a certain level to a region which is not covered by the max-stable model, $G$, for the upper tail of $F$ in \eqref{eq:FnG}.

\section{Copula and tail dependence coefficients}
\label{sec:cop}

Like any multivariate distribution function, a GP distribution $H$ with margins $H_1, \ldots, H_d$ has a copula $C : [0, 1]^d \to [0, 1]$, i.e., a distribution function with standard uniform margins such that
\[
  H( \bx ) = C \{ H_1(x_1), \ldots, H_d(x_d) \}, \qquad \bx \in \reals^d.
\]
Since $H_j$ is continuous on $(\eta_j, \infty)$, with $j = 1, \ldots, d$ and $\eta_j$ as in \eqref{eq:eta_j}, the copula $C$ is unique on $\prod_{j=1}^d [ H_j(\eta_j), 1 ]$ \citep{Sklar59, Nelsen06}. In addition, copulas remain invariant under increasing, componentwise transformations, so that by Proposition~\ref{prop:standard}, the set of copulas associated to $\GP( \bsigma, \bgamma, \bpi, \ell )$ does not depend on $(\bsigma, \bgamma)$ but only on $(\bpi, \ell)$. Explicit computation of the marginal quantile functions from Proposition~\ref{prop:margins} seems unfeasible, however. Still, we have the following, partial result. The tail copula \citep{schmidt+s:2006, einmahl+d+l:2006}, $R : [0, \infty)^d \to [0, \infty)$, associated to a stdf $\ell$ is defined by
\[
  R( \by )
  =
  \Lambda( [0, y_1] \times \cdots \times [0, y_d ] )
\]
where $\Lambda$ is the unique measure on $[0, \infty]^d \setminus \{ \binfty \}$ determined by $\ell$ via
\[
  \Lambda( [\bzero, \binfty] \setminus [\by, \binfty] )
  =
  \ell( \by ).
\]
The function $R$ can be expressed directly in terms of the function $\ell$ by the inclusion--exclusion formula. For instance, for $d=2$ we have $R(y_1, y_2) = y_1 + y_2 - \ell(y_1, y_2)$, whereas for general dimension $d$, we have
\begin{equation}
\label{eq:ell2R}
  R( \by ) 
  = 
  \sum_{J : \varnothing \ne J \subset \{1,\ldots,d\}} 
  (-1)^{\lvert J \rvert - 1} \, 
  \ell \bigl( (y_j \1_{\{j \in J\}})_{j=1}^d \bigr).
\end{equation}
A more insightful formula is that if there exists a random vector $\bV$ on $[0, \infty)^d$ with $\expec[ V_j ] = 1$ for all $j = 1, \ldots, d$ such that \eqref{eq:V2ell} holds, then the tail copula, $R$, associated to $\ell$ is given by
\begin{equation}
\label{eq:V2R}
  R( \by ) = \expec [ \min( \by \bV ) ],
  \qquad \by \in [0, \infty)^d.
\end{equation}
Equation~\eqref{eq:V2R} follows from \eqref{eq:V2ell} and \eqref{eq:ell2R} and the minimum--maximum identity. Identity \eqref{eq:V2R} can be applied to either $V_j = e^{S_j} / \expec[ e^{S_j} ]$ in \eqref{eq:S2ell} or to $V_j = e^{U_j} / \expec[ e^{U_j} ]$ in \eqref{eq:U2ell}.

\begin{proposition}
\label{prop:copula}
If $\bX \sim \GP( \bsigma, \bgamma, \bpi, \ell )$, then, for any $\bx \in [0, \infty)^d$, we have
\begin{align}
\label{eq:copula:ell}
  \Pr( \exists j = 1, \ldots, d : X_j > x_j )
  &=
  \ell \{ \Pr( X_1 > x_1 ), \ldots, \Pr( X_d > x_d ) \}, \\
\label{eq:copula:R}
  \Pr( \forall j = 1, \ldots, d : X_j > x_j )
  &=
  R \{ \Pr( X_1 > x_1 ), \ldots, \Pr( X_d > x_d ) \},
\end{align}
where $R$ is the tail copula associated to $\ell$.
\end{proposition}


Equation~\eqref{eq:copula:R} says that, on the set $\prod_{j=1}^d [0, \Hbar_j(0)]$, the survival copula of $H = \GP( \bsigma, \bgamma, \bpi, \ell )$ is given by the tail copula associated to $\ell$.

The functions $\ell$ and $R$ are homogeneous. As a consequence, for $p \ge 0$ such that $p \le \Hbar_j(0)$ for all $j = 1, \ldots, d$, we have
\begin{align*}
  \Pr\{ \exists j = 1, \ldots, d : \Hbar_j(X_j) < p \}
  &=
  p \, \ell(1, \ldots, 1), 
  \\
  \Pr\{ \forall j = 1, \ldots, d : \Hbar_j(X_j) < p \}
  &= p \, R(1, \ldots, 1).
\end{align*}
The quantity $\ell(1, \ldots, 1)$ is an example of an extremal coefficient \citep{schlather+t:2002}, whereas the quantity $R(1, \ldots, 1)$ is an example of a multivariate tail dependence coefficient \citep{schmid+etal:2010}.
One way to exploit the above relations in model checking is to check whether the ratio of marginal and joint exceedance probabilities is indeed constant starting from a certain point on, either via diagnostic plots or via formal hypothesis tests; \changed{see for instance \cite{kiriliouk+rootzen+segers+wadsworth:2016}.}

\begin{remark}[Other tail dependence functions]
The above formulas involving $\ell$ and $R$ could be generalized to more general exceedance events involving exceedances in some components and non-exceedances in some other components, perhaps using additional conditioning \citep{joe+l+n:2010}.
\end{remark}

\begin{remark}[Nonparametric inference]
By homogeneity, the stdf $\ell$ is determined by its values on $[0, \eps]^d$ for arbitrarily small, positive $\eps > 0$. Relation~\eqref{eq:copula:ell} could then serve as a basis for nonparametric inference on $\ell$, for instance via the empirical copula or the empirical stable tail dependence function \citep{einmahl2012}.
\end{remark}

\section{Stability}
\label{sec:stable}

Lower-dimensional margins of GP distributions, conditionally on having at least one positive component, are GP distributed as well. In addition, conditional distributions of multivariate threshold excesses by GP random vectors are GP distributed too \citep{rootzen2006}. These two properties are expressed together in the following result. For a vector $\bx \in (0, \infty)^d$ and a non-empty subset $J$ of $\{1, \ldots, d\}$, let $\bx_J$ denote $(x_j)_{j \in J}$. Further, for a $d$-variate stdf $\ell$, let $\ell_J$ denote the function $[0, \infty)^J \to [0, \infty)$ given by
\[
  \textstyle
  \ell_J( \by ) = \ell \bigl( \sum_{j \in J} y_j \bm{e}_j \bigr),
\]
where $\bm{e}_1, \ldots, \bm{e}_d$ are the $d$ canonical unit vectors in $\reals^d$. The function $\ell_J$ is a $\abs{J}$-variate stdf too; in fact, if $\ell$ is the stdf of the GEV $G$, then $\ell_J$ is the stdf of the $J$-margin, $G_J$, of $G$.

\begin{proposition}
\label{prop:stable}
Let $\bX \sim \GP( \bsigma, \bgamma, \bpi, \ell )$, let $\varnothing \ne J \subset \{1, \ldots, d\}$, and let $\bu \in [0, \infty)^J$ be such that $\Pr[ X_j > u_j ] > 0$ for all $j \in J$. Then
\begin{equation}
\label{eq:XJ|u:GP}
  \law ( \bX_J - \bu \mid \bX_J \not\le \bu )
  =
  \GP \left( 
    \bsigma_J + \bgamma_J \bu, \,
    \bgamma_J, \,
    (\Pr[ X_j > u_j \mid \bX_J \not\le \bu ])_{j \in J}, \,
    \ell_J
  \right).
\end{equation}
\end{proposition}

Here are some interesting special cases:
\begin{itemize}
\item
The special case $J = \{j\}$ reproduces the result that the conditional distribution of $X_j$ given that $X_j > 0$ is univariate GP with parameters $(\gamma_j, \sigma_j)$, a fact we already knew from Proposition~\ref{prop:H:identify}.

\item
If $u_j = 0$ for all $j \in J$, then we find that
\begin{equation}
\label{eq:X|0:J}
  \law \{ \bX_J \mid \max( \bX_J ) > 0 \}
  =
  \GP \left(
    \bsigma_J, \,
    \bgamma_J, \,
    \bpi_J / \ell_J( \bpi_J ), \,
    \ell_J
  \right).
\end{equation}

\item
The most compact formula arises in the $\GPU$ representation: if $\bX \sim \GPU( \bsigma, \bgamma, \law(\bU) )$, then, by the previous equation and equations~\eqref{eq:U2pi}--\eqref{eq:U2ell}, we have
\[
  \law \{ \bX_J \mid \max( \bX_J ) > 0 \}
  =
  \GPU \left(
    \bsigma_J, \,
    \bgamma_J, \,
    \law( \bU_J )
  \right).
\]
To see this, calculate \eqref{eq:U2pi}--\eqref{eq:U2ell} with $\bU$ replaced by $\bU_J$ and check that the results are equal to $\bpi_J / \ell_J( \bpi_J )$ and $\ell_J$, respectively, which is sufficient by \eqref{eq:X|0:J}.
\end{itemize}

\section{Linear combinations}
\label{sec:lincomb}

We investigate the joint conditional distribution of linear combinations of the components of a GP random vector with nonnegative coefficients in case all the shape parameters $\gamma_j$ are identical. To express the formulas compactly, we use matrix notation. The random vector $\bX$ and the scale parameter vector $\bsigma$ are seen as $d \times 1$ column vectors. The $i$-th row of the $m \times d$ matrix $\bA$ is denoted by $\bA_i$. Expressions such as $\bA \bX$ and $\bA_i \bX$ are to be interpreted via matrix products.

If $\gamma_j \le 0$, we need to take into account the possibility of masses as $-\infty$. Therefore, we apply the convention that in expressions like $\sum_{j=1}^d a_j X_j$, if $a_j = 0$ and $X_j = -\infty$, then $0 \times (-\infty)$ is to be interpreted as $0$, i.e., the $j$-th component of $\bX$ was not `selected' in the first place: $\sum_{j=1}^d a_j X_j = \sum \{ a_j X_j : a_j \ne 0 \}$.

\begin{proposition}
\label{prop:lincomb}
Let $\bX \sim \GPS( \bsigma, \bgamma, \law(\bS) )$ be such that $\gamma_1 = \ldots = \gamma_d =: \gamma$. Further, let $\bA = (a_{i,j})_{i,j} \in [0, \infty)^{m \times d}$ be such that $\Pr[ \bA_i \bX > 0 ] > 0$ for all $i = 1, \ldots, m$. For $\bx \in \reals^m$ such that $\bA_i \bsigma + \gamma x_i > 0$ for all $i = 1, \ldots, m$, we have
\begin{equation}
\label{eq:PAX}
  \Pr[ \bA \bX \not\le \bx ]
  =
  \expec \left[
    1 
    \wedge 
    \max_{i=1, \ldots, m} 
      \{ (1 + \gamma x_i / \bA_i \bsigma)^{-1/\gamma} \, e^{U_i} \}
  \right]
\end{equation}
where $\bU = (U_1, \ldots, U_m)$ is given by
\begin{equation}
\label{eq:lincomb:Ui}
  U_i
  =
  \begin{cases}
    \gamma^{-1} \log \bigl( \sum_{j=1}^d p_{i,j} \, e^{\gamma S_j} \bigr)
    & \text{if $\gamma \ne 0$,} \\[1ex]
    \sum_{j=1}^d p_{i,j} \, S_j
    & \text{if $\gamma = 0$,}
  \end{cases}
\end{equation}
where $p_{i,j} = a_{i,j} \sigma_j / \bA_i \bm{\sigma}$. As a consequence,
\[
  \law( \bA \bX \mid \bA \bX \not\leq \bzero )
  =
  \GPU \left(
    \bA \bsigma, \,
    \bgamma, \,
    \law( \bU )
  \right).
\]
\end{proposition}

If $m = 1$ and $\bA = \bm{a} \in [0, \infty)^{1 \times d}$, then Proposition~\ref{prop:lincomb} says that the law of $\bm{a} \bm{X}$ conditionally on $\bm{a} \bm{X} > 0$ is univariate GP with parameters $(\gamma, \bm{a} \bsigma)$. Note that the gist of Proposition~\ref{prop:lincomb} does not depend on the way in which the GP is parametrized: the only condition is that all marginal shape parameters be the same.

\appendix

\section{Proofs}
\label{app:proofs}

\subsection*{Proofs for Section~\ref{sec:param}}

\begin{proof}[Proof of Proposition~\ref{prop:G2H}]
For each $j = 1, \ldots, d$, the lower endpoint, $\eta_j$, of $G_j$ is given by $\eta_j = - \sigma_j/\gamma_j$ if $\gamma_j > 0$ and $\eta_j = -\infty$ if $\gamma_j \le 0$. It follows that $x_j > \eta_j$ if and only if $\sigma_j + \gamma_j x_j > 0$. Furthermore, by \eqref{eq:Gj}, we have $- \log G_j(0) = \tau_j \in (0, \infty)$, so that indeed $0 < G_j(0) < 1$, as required in the definition of $H = \GP(G)$.

By \eqref{eq:G2H} and \eqref{eq:ell2G}, we find, for such $\bx$,
\[
  H( \bx )
  =
  \frac%
    {%
      \ell( - \log G_1(x_1 \wedge 0), \ldots, - \log G_d(x_d \wedge 0) )
      -
      \ell( - \log G_1(x_1), \ldots, - \log G_d(x_d))
    }%
    {%
      \ell( - \log G_1(0), \ldots, - \log G_d(0) )
    }.
\]
Furthermore, by \eqref{eq:Gj}, we have
\begin{align*}
  - \log G_j(x_j)
  &=
  \{ 1 + \gamma_j (x_j - \mu_j) / \alpha_j \}^{-1/\gamma_j} \\
  &=
  (1 - \gamma_j \mu_j / \alpha_j)^{-1/\gamma_j} \,
  ( 1 + \gamma_j x_j / \sigma_j)^{-1/\gamma_j}.
\end{align*}
Combine these two equations to arrive at
\[
  H( \bx )
  =
  \frac%
    {%
      \ell( \btau (1 + \bgamma (\bx \wedge \bzero) / \bsigma)^{-1/\bgamma} )
      -
      \ell( \btau (1 + \bgamma \bx / \bsigma)^{-1/\bgamma} )
    }%
    {%
      \ell( \btau )
    }.
\]
Finally, use homogeneity of $\ell$ in \eqref{eq:ell} to arrive at \eqref{eq:H:pi-ell}.
\end{proof}

\begin{proof}[Proof of Proposition~\ref{prop:H:identify}]
Since $\bpi = \btau / \ell( \btau )$, we have $\ell( \bpi ) = 1$ by homogeneity of $\ell$ in \eqref{eq:ell}. Let $\omega_j \in (0, \infty]$ denote the upper endpoint of $G_j$ in Proposition~\ref{prop:G2H}; we have $\omega_j = \infty$ if $\gamma_j \ge 0$ and $\omega_j = \sigma_j / \abs{\gamma_j}$ if $\gamma_j < 0$. In \eqref{eq:H:pi-ell}, fix $j = 1, \ldots, d$ and $x_j \in [0, \omega_j)$ and let $x_k \to \omega_k$ for $k \in \{1, \ldots, d\} \setminus \{j\}$. Then $(1 + \gamma_k x_k / \sigma_k)^{-1/\gamma_k} \to 0$, so that, by $\ell( \bpi ) = 1$ and \eqref{eq:ell}, we find
\begin{equation}
\label{eq:Hjxj:0}
  H_j(x_j)
  =
  1 - \pi_j (1 + \gamma_j x_j / \sigma_j)^{-1/\gamma_j},
  \qquad x_j \in [0, \omega_j).
\end{equation}
This yields both \eqref{eq:Hj0} and \eqref{eq:Hj|0}. Combine \eqref{eq:H:pi-ell}, \eqref{eq:Hjxj:0} and $\ell(\bpi) = 1$ to arrive at \eqref{eq:Hbar:ell}.
\end{proof}

\subsection*{Proofs for Section~\ref{sec:repr}}

\begin{proof}[Proof of Proposition~\ref{prop:standard}]
Writing $\bZ = \bgamma^{-1} \log( 1 + \bgamma \bX / \bsigma )$, we have $\Pr[ \bZ \le \bz ] = \Pr[ \bX \le \bsigma (e^{\bgamma \bz} - 1) / \bgamma ]$, which, by \eqref{eq:H:pi-ell}, yields \eqref{eq:H:Z}. By \eqref{eq:H:pi-ell} applied to $( \bgamma, \bsigma ) = ( \bzero, \bone )$, we find that the right-hand side of \eqref{eq:H:Z} is indeed the expression for the cdf of the $\GP( \bone, \bzero, \bpi, \ell )$ distribution.
\end{proof}

\begin{proof}[Proof of Theorem~\ref{thm:S:direct}]
By \eqref{eq:V2ell}, the function $\ell$ in \eqref{eq:S2ell} is indeed a stdf. Since $\mS = 0$ almost surely, we also have $\ell( \bpi ) = \expec[ e^{\mS} ] = 1$. Clearly, $\max( \bS + E ) = E > 0$ almost surely. For $\bz \in \reals^d$ such that $\max( \bz ) > 0$, we have, since $e^{\mS} = 1$ almost surely,
\begin{align*}
  \Pr[ \bS + E \le \bz ]
  &=
  \Pr[ E \le \min( \bz - \bS ) ] \\
  &=
  1 - \expec[ \min\{ 1, e^{\max( \bS - \bz )} \} ] \\
  &=
  \expec[ \max \{0, 1 - e^{\max( \bS - \bz )} \} ] \\
  &=
  \expec[ \max \{ 0, e^{\mS} - e^{\max( \bS - \bz )} \} ] \\
  &=
  \expec[ e^{\max( \bS - (\bz \wedge \bzero) )} - e^{\max( \bS - \bz )} ].
\end{align*}
The last step can be seen through a case-by-case analysis. Now insert the expressions for $\bpi$ and $\ell$ in \eqref{eq:S2pi}--\eqref{eq:S2ell} into the right-hand side of \eqref{eq:H:Z} to see that $\Pr[ \bS + E \le \bz ] = H(\bz)$ with $H = \GP( \bone, \bzero, \bpi, \ell )$.
\end{proof}

\begin{proof}[Proof of Theorem~\ref{thm:S:converse}]
The uniqueness of the distribution of $\bS$ follows from the fact that $\bS$ can be recovered from $\bS + E$ via $\bS = \bS + E - \max(\bS + E)$, since $\max(\bS) = 0$ by assumption. We show the existence of $\bS$ with the required properties.

Let $\bV$ be a random vector with values $[0, \infty)^d$ such that $\expec[ V_j ] = 1$ for all $j = s1, \ldots, d$ and such that \eqref{eq:V2ell} holds. Let $\bW = \bpi \bV$ and define $\bS$ (or rather its distribution) through
\begin{equation}
\label{eq:W2S}
  \Pr[ \bS \in \point ]
  =
  \frac%
    {\expec[ \1 \{ \log(\bW / \mW) \in \point \} \, \mW ] }%
    {\expec[ \mW ]}.
\end{equation}
Here we put $\log(0) = -\infty$. Although the probability of the event $\{ \mW = 0 \}$ could be positive under the original distribution of $\bW$, the probability is zero under the transformed probability measure $\max( \bw ) \, (\expec[ \mW ])^{-1} \, \Pr[ \bW \in \diff \bw ]$. Equation~\eqref{eq:W2S} can be written in terms of expectations of measurable functions $g$ as
\begin{equation}
\label{eq:W2S:g}
  \expec[ g( \bS ) ]
  =
  \frac%
    {\expec[ g( \log(\bW / \mW) ) \, \mW ]}%
    {\expec[ \mW ]}.
\end{equation}

The random vector $\bS$ is a spectral random vector: set $g(\bm{s}) = \1\{ \max( \bm{s} ) = 0 \}$ and $g(\bm{s}) = e^{s_j}$, respectively, in \eqref{eq:W2S:g} to obtain
\begin{align*}
  \Pr[ \mS = 0 ]
  &=
  \frac%
    {\expec[ \1 \{ \max(\log( \bW / \mW )) = 0 \} \, \mW ] }%
    {\expec[ \mW ]}
  = 1, \\
  \expec[ e^{S_j} ]
  &= \frac{\expec[ W_j ]}{\expec[ \mW ]}
  = \frac{\pi_j}{\ell( \bpi )}
  = \pi_j,
\end{align*}
so that both $\Pr[S_j > -\infty] > 0$ and \eqref{eq:S2pi} hold. Equation~\eqref{eq:S2ell} follows from setting $g(\bm{s}) = \max \{ (\by / \bpi) e^{\bm{s}} \}$ in \eqref{eq:W2S:g}.
\end{proof}

\begin{proof}[Proof of Proposition~\ref{prop:S:I}]
The statement that $\bS = \bT - \mT$ is a spectral random vector is trivial; note that the property that $\mT > - \infty$ almost surely guarantees that $S_j = T_j - \mT$ is well-defined, even if $T_j = -\infty$ can occur with positive probability. To arrive at \eqref{eq:T2pi}--\eqref{eq:T2ell}, just substitute $S_j = T_j - \mT$ into \eqref{eq:S2pi}--\eqref{eq:S2ell}.

To obtain \eqref{eq:T2H}, combine \eqref{eq:H:Z} with \eqref{eq:T2pi}--\eqref{eq:T2ell} and simplify, using the identities $\max \{ \bT - ( \bz \vee \bzero ) \} = \max( \bT - \bz ) \vee \mT$ and $a \vee b - a = b - b \wedge a$.
\end{proof}

\begin{proof}[Proof of Proposition~\ref{prop:S:II}]
Equation~\eqref{eq:U2S} implies that, for measurable functions $g$, we have
\begin{equation}
\label{eq:U2S:g}
  \expec[ g( \bS ) ]
  =
  \frac%
    {\expec[ g( \bU - \mU ) \, e^{\mU} ]}%
    {\expec[ e^{\mU} ]},
\end{equation}
in the sense that the expectation on the left-hand side is defined if and only if the one on the right-hand side is defined, in which case both sides of \eqref{eq:U2S:g} are equal. The random vector $\bS$ is indeed a spectral random vector:
\begin{align*}
  \Pr[ \mS = 0 ] 
  &= \expec[ \1 \{ \mU - \mU = 0 \} \, e^{\mU} ] / \expec[ e^{\mU} ] = 1, \\
  \expec[ e^{S_j} ] 
  &= \expec[ e^{U_j - \mU} \, e^{\mU} ] / \expec[ e^{\mU} ] 
  = \expec[ e^{U_j} ] / \expec[ e^{\mU} ] > 0. 
\end{align*}
Equations~\eqref{eq:U2pi}--\eqref{eq:U2ell} follow from combining \eqref{eq:S2pi}--\eqref{eq:S2ell} and \eqref{eq:U2S:g} with $g( \bS ) = e^{S_j}$ and $g( \bS ) = \max \{ (\by/\bpi) \, e^{\bS} \}$, respectively.

To obtain \eqref{eq:U2H}, combine \eqref{eq:H:Z} with \eqref{eq:U2pi}--\eqref{eq:U2ell} and simplify, using the identities $\max \{ \bU - ( \bz \vee \bzero ) \} = \max( \bU - \bz ) \vee \mU$ and $a \vee b - a = b - b \wedge a$.
\end{proof}

\begin{proof}[Proof of Proposition~\ref{prop:R}]
For $t \in (0, \infty)$ and for $\bx$ such that $x_j + \sigma_j / \gamma_j > 0$, we have
\begin{align*}
  \Fbar_{\bR}( t^{\bgamma}( \bx + \bsigma / \bgamma ) )
  &=
  \Pr[ \bR \not\le t^{\bgamma}( \bx + \bsigma / \bgamma ) ] \\
  &=
  \Pr \left[ 
    \max_{j=1,\ldots,d} 
    \left( \frac{R_j}{x_j + \sigma_j/\gamma_j} \right)^{1/\gamma_j} 
    > t 
  \right] \\
  &=
  \Pr \left[
    \max_{j=1,\ldots,d} 
    \{ (1 + \gamma_j x_j / \sigma_j)^{-1/\gamma_j} e^{U_j} \} 
    > t
  \right],
\end{align*}
with $U_j = \gamma_j^{-1} \log (\gamma_j R_j / \sigma_j)$ for all $j = 1,\ldots,d$. It follows that
\[
  \int_{t=0}^{\infty}
    \Fbar_{\bR}( t^{\bgamma}( \bx + \bsigma / \bgamma ) ) \,
  \diff t
  =
  \expec \left[
    \max_{j=1,\ldots,d} 
    \{ (1 + \gamma_j x_j / \sigma_j)^{-1/\gamma_j} e^{U_j} \}     
  \right].
\]
Apply this identity three times to the right-hand side of \eqref{eq:R2H} and compare the resulting expression with the right-hand side in \eqref{eq:U2H} to see that $H_R$ is indeed the cdf of the stated GP distribution.
\end{proof}

\subsection*{Proofs for Section~\ref{sec:pdf}}

The proofs of Lemma~\ref{lem:Z2X:pdf}, Theorem~\ref{thm:pdf:I}, Theorem~\ref{thm:pdf:II} and Corollary~\ref{cor:R:pdf} are given in Appendix~\ref{app:suppl}.

\begin{proof}[Proof of Theorem~\ref{thm:pdf:S}]
The cdf, $H$, of $\bZ$ is given by
\begin{align*}
  H(\bz) 
  &= \int_0^\infty \Pr(\bS+y \leq \bz) \, e^{-y} \, \diff y \\
  &= 
  \int_0^\infty 
    \left( 
      \sum_{j=1}^d 
      \int_{\Sj} 
	\1(\bs + y \leq \bz) \, 
      f_{\bS}( \bs) \, \diff \bs_{-j} 
    \right) \, 
  e^{-y} \, \diff y,
\end{align*}
where $\Sj = \{\bm{s}\in\mathbb{R}^d: s_j=\max(\bs)=0\}$, $\bs_{-j} \in \Sj$, and $\diff \bs_{-j}$ is the $(d-1)$-dimensional Lebesgue measure on $\Sj$. Now let $\Uj = \Sj + (0,\infty) =  \{\bm{u}\in\mathbb{R}^d: u_j=\max(\bu)>0\}$ and on $\Uj$ make the substitution $\bu_j := \bs_{-j} +y$, consisting of $(\bu_j)_j = y$ and $(\bu_j)_i = s_i+y$ for $i \neq j$, so that $y=\max(\bu_j)=(\bu_j)_j$. It is easily verified that the determinant of the Jacobian of this transformation is equal to one. Write $\diff \bu_j = \diff y \, \diff \bs_{-j}$, the $d$-dimensional Lebesgue measure on $\Uj$. By Fubini's theorem,
\begin{align*}
  H(\bz) 
  &= 
  \sum_{j=1}^d 
  \int_{\Sj} 
    \int_0^\infty 
      \1(\bs+y \leq \bz) \, f_{\bS}(\bs) \, e^{-y} \,
    \diff y \,
  \diff \bs_{-j} \\
  &=
  \sum_{j=1}^d 
  \int_{\Uj} 
    \1(\bu_j \leq \bz) \, f_{\bS}(\bu_j-\max(\bu_j)) \, e^{-\max(\bu_j)} \,
  \diff \bu_j \\
  &=
  \sum_{j=1}^d 
  \int_{\Uj \cap (-\bm{\infty},\bz]} 
    f_{\bS}(\bu_j-\max(\bu_j)) \, e^{-\max(\bu_j)} \,
  \diff \bu_j \\
  &=
  \int_{(-\bm{\infty},\bz]} 
    f_{\bS}(\bu-\max(\bu)) \, e^{-\max(\bu)} \,
  \diff \bu,
\end{align*}
with $\diff \bu$ the $d$-dimensional Lebesgue measure on $\bigcup_{j=1}^d \Uj = \{ \bm{u} \in \mathbb{R}^d : \max( \bm{u} ) > 0 \}$.
\end{proof}

\subsection*{Proofs for Section~\ref{sec:margin}}

\begin{proof}[Proof of Proposition~\ref{prop:margins}]
Straightforward from \eqref{eq:H:Z} and the representations of $(\bpi, \ell)$ in terms of $\law(\bS)$, $\law(\bT)$ and $\law(\bU)$.
\end{proof}

\begin{proof}[Proof of Proposition~\ref{prop:etaj}]
Equation~\eqref{eq:Hj:etaj:1} follows from \eqref{eq:Hj:z} since $y_j = e^{-z_j}$ converges to $\infty$ if $z_j$ tends to $-\infty$. Equation~\eqref{eq:Hj:etaj:2} then follows from equation~\eqref{eq:Hj:etaj:1} by setting $\eps = y_j^{-1}$ and using homogeneity from $\ell$. Equation~\eqref{eq:Hj:etaj:3} follows from equation~\eqref{eq:Hj:etaj:2}, the fact that $\ell( \pi_j \bm{e}_j ) = \pi_j$, and properties of directional derivatives. The first part of equation~\eqref{eq:Hj:etaj:STU} follows by taking the limit as $z_j \to -\infty$ in \eqref{eq:Hj:S} and applying the dominated convergence theorem together with the fact that $e^{\mS} = 1$ almost surely. The other two identities in \eqref{eq:Hj:etaj:STU} then follow from expressing the law of $\bS$ in terms of those of $\bT$ and $\bU$, respectively.
\end{proof}

\subsection*{Proofs for Section~\ref{sec:cop}}

\begin{proof}[Proof of Proposition~\ref{prop:copula}]
Equation~\eqref{eq:copula:ell} is the same as equation~\eqref{eq:Hbar:ell}. Equation~\eqref{eq:copula:R} then follows from equation~\eqref{eq:copula:ell} and the inclusion--exclusion formula.
\end{proof}

\subsection*{Proofs for Section~\ref{sec:stable}}

\begin{proof}[Proof of Proposition~\ref{prop:stable}]
For $\bx \in \reals^J$ such that $\max_{j \in J} x_j > 0$, we have,
\begin{align}
\nonumber
  \Pr[ \bX_J - \bu \le \bx \mid \bX_J \not\le \bu ]
  &=
  \frac%
    {\Pr[ \bX_J - \bu \le \bx, \bX_J \not\le \bu ]}%
    {\Pr[ \bX_J \not\le \bu ]} \\
\nonumber
  &=
  \frac%
    {\Pr[ \bX_J - \bu \le \bx] - \Pr[ \bX_J - \bu \le \bx, \bX_J \le \bu ]}%
    {\Pr[ \bX_J \not\le \bu ]} \\
\label{eq:XJ|u:cdf}
  &=
  \frac%
    {\Pr[ \bX_J \le \bu + \bx ] - \Pr[ \bX_J \le \bu + (\bx \wedge \bzero) ]}%
    {\Pr[ \bX_J \not\le \bu ]}.
\end{align}
Since $\bu \ge \bzero$, we have
\[
  \Pr[ \bX_J \not\le \bu ]
  =
  \ell_J \left( (\Pr[ X_j > u_j ])_{j \in J} \right)
  =
  \ell_J \left( \bpi_J (1 + \bgamma_J \bu / \bsigma_J )^{-1/\bgamma_J} \right);
\]
to see this, let $u_k \to \infty$ for $k \in \{1, \ldots, d\} \setminus J$ in \eqref{eq:Hbar:ell}. 
In addition, the $J$-margin of $H$ is given by
\begin{equation}
\label{eq:XJ:cdf}
  \Pr[ \bX_J \le \bv ]
  =
  \ell \left( 
    \bpi \, ( 1 + \bgamma (\bw \wedge \bzero) / \bsigma )^{-1/\bgamma} 
  \right)
  -
  \ell_J \left( 
    \bpi_J \, ( 1 + \bgamma_J \bv / \bsigma_J )^{-1/\bgamma_J} 
  \right)  
\end{equation}
for $\bv \in \reals^J$ such that $\sigma_j + \gamma_j v_j > 0$ for all $j \in J$ and where $\bw \in \reals^d$ is defined by $w_j = v_j$ if $j \in J$ and $w_j = 0$ otherwise; this follows from \eqref{eq:H:pi-ell}. Substitute \eqref{eq:XJ:cdf} for $\bv = \bu + \bx$ and $\bv = \bu + (\bx \wedge \bzero)$ into \eqref{eq:XJ|u:cdf}. The numerator will have four instances of $\ell$, two of which will cancel out because $( \bu + (\bx \wedge \bzero) ) \wedge \bzero = ( \bu + \bx ) \wedge \bzero$, a consequence of the assumption that $\bu \ge \bzero$. It follows that
\begin{multline*}
  \Pr[ \bX_J - \bu \le \bx \mid \bX_J \not\le \bu ] \\
  =
  \frac%
    {%
      \ell_J \left( 
	\bpi_J \, ( 1 + \bgamma_J \{ \bu + (\bx \wedge \bzero) \} / \bsigma_J )^{-1/\bgamma_J} 
      \right)
      -
      \ell_J \left( 
	\bpi_J \, ( 1 + \bgamma_J \{ \bu + \bx \} / \bsigma_J )^{-1/\bgamma_J} 
      \right)
    }%
    {%
      \ell_J \left( 
	\bpi_J \, (1 + \bgamma_J \bu / \bsigma_J)^{-1/\bgamma_J} 
      \right)
    }.
\end{multline*}
In addition, note that, for all $j \in J$,
\begin{align*}
  (1 + \gamma_j (u_j + y_j) / \sigma_j )^{-1/\gamma_j}
  &=
  (1 + \gamma_j u_j/\sigma_j)^{-1/\gamma_j} 
  (1 + \gamma_j y_j / (\sigma_j + \gamma_j u_j))^{-1/\gamma_j},
\end{align*}
where $y_j$ represents either $x_j$ or $x_j \wedge 0$. Writing
\[
  \btau_J = \bpi_J \, (1 + \bgamma_J \bu / \bsigma_J)^{-1/\bgamma_J},
\]
we find
\begin{multline*}
  \Pr[ \bX_J - \bu \le \bx \mid \bX_J \not\le \bu ] \\
  =
  \frac%
    {%
      \ell_J \left( 
	\btau_J (1 + \bgamma_J (\bx \wedge \bzero) / (\bsigma_J + \bgamma_J \bu) )^{-1/\bgamma_J}
      \right)
      -
      \ell_J \left(
	\btau_J (1 + \bgamma_J \bx / (\bsigma_J + \bgamma_J \bu) )^{-1/\bgamma_J}
      \right)
    }%
    {\ell_J( \btau_J )}.
\end{multline*}
Since $\tau_j = \Pr[ X_j > u_j ]$ and $\ell_J( \btau_J ) = \Pr[ \bX_J \not\le \bu ]$ and thus $\tau_j / \ell_J( \btau_J ) = \Pr[ X_j > u_j \mid \bX_J \not\le \bu ]$, we obtain \eqref{eq:XJ|u:GP}.
\end{proof}

\subsection*{Proofs for Section~\ref{sec:lincomb}}

\begin{proof}[Proof of Proposition~\ref{prop:lincomb}]
By definition, $\bS$ is a spectral random vector (Definition~\ref{def:S}) and by \eqref{eq:Z2X} we have
\[
  \bX \eqd 
  \begin{cases}
    \bsigma (e^{\gamma( \bS + E )} - 1) / \gamma,
    & \text{if $\gamma \ne 0$,} \\[1ex]
    \bsigma (\bS + E),
    & \text{if $\gamma = 0$,}
  \end{cases}
\]
where $E$ is a unit exponential random variable, independent of $\bS$. Suppose $\gamma > 0$. Then
\begin{align*}
  \bA \bX \not\le \bx
  &\iff \exists i = 1, \ldots, m : 
    \sum_{j=1}^d a_{i,j} \sigma_j \frac{e^{\gamma(S_j + E)} - 1}{\gamma} > x_j 
  \\
  &\iff \exists i = 1, \ldots, m :
    e^{\gamma E} \sum_{j=1}^d a_{i,j} \sigma_j e^{\gamma S_j}
    - \sum_{j=1}^d a_{i,j} \sigma_j > \gamma x_j
  \\
  &\iff \exists i = 1, \ldots, m :
    \left( 
      \frac%
	{\sum_{j=1}^d a_{i,j} \sigma_j e^{\gamma S_j}}%
	{\sum_{j=1}^d a_{i,j} \sigma_j + \gamma x_j}
    \right)^{1/\gamma}
    > e^{-E}.
\end{align*}
The random variable $e^{-E}$ is independent of $\bS$ and is uniformly distributed on the interval $[0, 1]$. We find
\begin{equation}
\label{eq:PAX:nonzero}
  \Pr[ \bA \bX \not\le \bx ]
  =
  \expec \left[
    1 \wedge
    \max_{i=1,\ldots,m}
    \left( 
      \frac%
	{\sum_{j=1}^d a_{i,j} \sigma_j e^{\gamma S_j}}%
	{\sum_{j=1}^d a_{i,j} \sigma_j + \gamma x_j}
    \right)^{1/\gamma}
  \right].
\end{equation}
If $\gamma < 0$, then a similar argument yields the same expression, while if $\gamma = 0$, we can apply a similar reasoning to find that
\begin{equation}
\label{eq:PAX:zero}
  \Pr[ \bA \bX \not\le \bx ]
  =
  \expec \left[
   1 \wedge   
   \max_{i=1,\ldots,m}
   \exp \left(
    \frac{ \sum_{j=1}^d a_{i,j} \sigma_j S_j }{ \sum_{j=1}^d a_{i,j} \sigma_j }
    -
    \frac{ x_i }{ \sum_{j=1}^d a_{i,j} \sigma_j }
   \right)
  \right].
\end{equation}
Since $\sum_{j=1}^d a_{i,j} \sigma_j = \bA_i \bsigma$, the two expressions for $\Pr[ \bA \bX \not\le \bx ]$ in \eqref{eq:PAX:nonzero}--\eqref{eq:PAX:zero} are equal to the one claimed in \eqref{eq:PAX} with $U_i$ given by \eqref{eq:lincomb:Ui}.

Next we compute the conditional distribution of $\bA \bX$ given that $\bA \bX \not\le \bzero$. For $\bx \in \reals^d$, we have, by a computation similar as the one leading to \eqref{eq:XJ|u:cdf},
\begin{align*}
  \Pr[ \bA \bX \le \bx \mid \bA \bX \not\le \bzero ]
  &=
  \frac%
    {\Pr[ \bA \bX \le \bx, \, \bA \bX \not\le \bzero ]}%
    {\Pr[ \bA \bX \not\le \bzero ]} \\
  &=
  \frac%
    {\Pr[ \bA \bX \le \bx ] - \Pr[ \bA \bX \le \bx \wedge \bzero ]}%
    {\Pr[ \bA \bX \not\le \bzero ]} \\
  &=
  \frac%
    {\Pr[ \bA \bX \not\le \bx \wedge \bzero ] - \Pr[ \bA \bX \not\le \bx ]}%
    {\Pr[ \bA \bX \not\le \bzero ]}.
\end{align*}
For $\bx$ such that $\bA_i \bsigma + \gamma x_i > 0$ for all $i = 1, \ldots, m$, the three probabilities can be worked out using \eqref{eq:PAX}. Regarding the denominator: since $U_i \le 0$ almost surely, we find
\[
  \Pr[ \bA \bX \not\le \bzero ]
  =
  \expec[ e^{\mU} ].
\]
Regarding the numerator: apply \eqref{eq:PAX} twice, to $\bx \wedge \bzero$ and to $\bx$ itself. We find
\begin{align*}
  \lefteqn{
    \Pr[ \bA \bX \not\le \bx \wedge \bzero ] - \Pr[ \bA \bX \not\le \bx ]
  } \\
  &=
  \expec \left[
    1 
    \wedge 
    \max_{i=1, \ldots, m} 
      \{ (1 + \gamma (x_i \wedge 0) / \bA_i \bsigma)^{-1/\gamma} \, e^{U_i} \}
  \right]
  -
  \expec \left[
    1 
    \wedge 
    \max_{i=1, \ldots, m} 
      \{ (1 + \gamma x_i / \bA_i \bsigma)^{-1/\gamma} \, e^{U_i} \}
  \right] \\
  &=
  \expec \left[
    \max_{i=1, \ldots, m} 
      \{ (1 + \gamma (x_i \wedge 0) / \bA_i \bsigma)^{-1/\gamma} \, e^{U_i} \}
  \right]
  -
  \expec \left[
    \max_{i=1, \ldots, m} 
      \{ (1 + \gamma x_i / \bA_i \bsigma)^{-1/\gamma} \, e^{U_i} \}
  \right].
\end{align*}
The reason we may omit the two instances of ``$1 \wedge \ldots$'' is again because $U_i \le 0$ almost surely. The identity can be confirmed by a case-by-case analysis.

Comparing the resulting expression for $\Pr[ \bA \bX \le \bx \mid \bA \bX \not\le \bzero ]$ with \eqref{eq:U2H} confirms that the law of $\bA \bX$ given that $\bA \bX \not\le \bzero$ is given by the $\GPU( \bgamma, \bA \bsigma, \law( \bU ))$ distribution.  
\end{proof}

\section{Supplementary proofs}
\label{app:suppl}

\begin{proof}[Proof of Lemma~\ref{lem:Z2X:pdf}]
Let $\bx \in \reals^d$ be such that $\bx \not\le \bzero$ and $\sigma_j + \gamma_j x_j > 0$ for all $j = 1, \ldots, d$. Let $\bz = \bgamma^{-1} \log( 1 + \bgamma \bx / \bsigma )$. Then
\begin{equation*}
  h_{\bX} ( \bx )
  =
  h_{\bZ} ( \bz ) \, 
  \prod_{j=1}^d \frac{\diff z_j}{\diff x_j}
  =
  h_{\bZ} \left( \bgamma^{-1} \log( 1 + \bgamma \bx / \bsigma ) \right) \,
  \prod_{j=1}^d \frac{1}{\sigma_j + \gamma_j x_j}. \qedhere
\end{equation*}
\end{proof}

\begin{proof}[Proof of Theorem~\ref{thm:pdf:I}]
Let $\bS = \bT - \mT$ and let $E$ be a unit exponential random variable, independent of $\bT$. By definition, $H$ is the cdf of $\bZ = \bS + E = \bT - \mT + E$, so that
\begin{align*}
  H( \bz )
  &=
  \Pr[ \bT - \mT + E \le \bz ] \\
  &=
  \int_0^\infty \Pr[ \bT - \mT + y \le \bz ] \, e^{-y} \, \diff y \\
  &=
  \int_{\reals^d} 
    \int_0^\infty 
      \1 \{ \bt - \mt + y \le \bz \} \, 
      f_{\bT}( \bt ) \, 
    e^{-y} \, \diff y \,
  \diff \bt.
\end{align*}
In the inner integral, perform the substitution $\mt - y = r$ to see that
\begin{align*}
  H( \bz )
  &=
  \int_{\reals^d} 
    \int_{-\infty}^{\mt}
      \1 \{ \bt - r \le \bz \} \,
      f_{\bT}( \bt ) \,
    e^{r-\mt} \, \diff r \,
  \diff \bt.
\end{align*}
Next, apply Fubini's theorem and the substitutions $t_j - r = u_j$ for $j = 1, \ldots, d$ to see that
\begin{align*}
  H( \bz )
  &=
  \int_{\reals}
    \int_{\reals^d}
      \1 \{ r \le \max( \bt ), \; \bt - r \le \bz \} \,
      f_{\bT}( \bt ) \,
      e^{r-\mt} \,
    \diff \bt \,
  \diff r \\
  &=
  \int_{\reals}
    \int_{\reals^d}
      \1 \{ \max(\bu) \ge \bzero, \; \bu \le \bz \} \,
      f_{\bT}( \bu + r ) \,
      e^{-\max(\bu)} \,
    \diff \bu \,
  \diff r \\
  &=
  \int_{\bu \in (-\binfty, \bz]}
    \1 \{ \max(\bu) \ge 0 \} \, 
    e^{-\max( \bu )}
    \int_{r \in \reals}
      f_{\bT}( \bu + r ) \,
    \diff r \,
  \diff \bu.
\end{align*}
Finally, substitute $r = \log(t)$ to find
\[
  H( \bz )
  =
  \int_{\bu \in (-\binfty, \bz]}
    \1 \{ \max(\bu) \ge 0 \} \, 
    e^{-\max( \bu )}
    \int_{t = 0}^\infty
      f_{\bT}( \bu + \log t ) \,
    t^{-1} \, \diff t \,
  \diff \bu.
\]
We obtain that $H( \bz ) = \int_{(-\binfty, \bz]} h( \bu ) \, \diff \bu$ with $h$ given by \eqref{eq:T2h}.
\end{proof}

\begin{proof}[Proof of Theorem~\ref{thm:pdf:II}]
By definition, the function $H$ is the cdf of the random vector $\bS + E$, with $E$ a unit exponential random variable, independent of the random vector $\bS$, the distribution of which is determined by the one of $\bU$ through equation~\eqref{eq:U2S}. By repeated applications of Fubini's theorem and by appropriate changes of variables, we find
\begin{align*}
  H( \bz )
  &=
  \Pr[ \bS + E \le \bz ] \\
  &=
  \int_0^\infty \Pr[ \bS + y \le \bz ] \, e^{-y} \, \diff y \\
  &=
  \frac{1}{\expec[ e^{\mU} ]}
  \int_0^\infty \expec[ \1 \{ \bU - \mU + y \le \bz \} \, e^{\mU} ] \, e^{-y} \, \diff y \\
  &=
  \frac{1}{\expec[ e^{\mU} ]}
  \int_0^\infty 
    \int_{\reals^d} 
      \1 \{ \bu - \max(\bu) + y \le \bz \} \, 
      e^{\max(\bu)-y} \, 
      f_{\bU}( \bu ) \,
    \diff \bu \,
  \diff y \\
  &=
  \frac{1}{\expec[ e^{\mU} ]}
  \int_{\reals^d} 
    \int_{\reals}
      \1 \{ \bu - s \le \bz, \, s < \max(\bu) \} \, 
      e^{s} \, 
      f_{\bU}( \bu ) \,
    \diff s \,
  \diff \bu \\
  &=
  \frac{1}{\expec[ e^{\mU} ]}
  \int_{\reals^d}
    \int_{\reals}
      \1 \{ \bv \le \bz, \, \max( \bv ) > 0 \} \,
      e^s \,
      f_{\bU}( \bv + s ) \,
    \diff s \,
  \diff \bv \\
  &=
  \frac{1}{\expec[ e^{\mU} ]}
  \int_{\bv \in (-\binfty, \bz]}
    \1 \{ \bv \not\le 0 \}
    \int_{\reals}
      e^s \,
      f_{\bU}( \bv + s ) \,
    \diff s \,
  \diff \bv.
\end{align*}
Apply the substitution $e^s = t$ to see that $H( \bz ) = \int_{(-\binfty, \bz]} h( \bv ) \, \diff \bv$ with $h$ given by \eqref{eq:U2h}.
\end{proof}

\begin{proof}[Proof of Corollary~\ref{cor:R:pdf}]
Let $\bU = \bgamma^{-1} \log( \bgamma \bR / \bsigma )$. Let $\bu \in \reals^d$ and let $\bm{r} = (\bsigma/\bgamma) e^{\bgamma \bu}$. The density function, $f_{\bU}$, of $\bU$ is given by
\[
  f_{\bU}( \bu )
  =
  f_{\bR}( \bm{r} ) \, \prod_{j=1}^d \frac{\diff r_j}{\diff u_j}
  =
  f_{\bR} \left( (\bsigma/\bgamma) e^{\bgamma \bu} \right) \, 
  \prod_{j=1}^d \sigma_j e^{\gamma_j u_j}.
\]
By Theorem~\ref{thm:pdf:II}, the density function, $h_{\bZ}$, of $\bZ \sim \GPU( \bzero, \bone, \law( \bU ))$ is then given by
\begin{align*}
  h_{\bZ}( \bz )
  &=
  \1 ( \bz \not\le \bzero )
  \frac{1}{ \expec[ \max \{ (\bgamma \bR / \bsigma)^{1/\bgamma} \} ] }
  \int_0^\infty
    f_{\bR} \bigl( (\bsigma/\bgamma) \, e^{\bgamma (\bz + \log t)} \bigr) \, 
    \prod_{j=1}^d \sigma_j \, e^{\gamma_j (z_j + \log t)} \,
  \diff t \\
  &=
  \1 ( \bz \not\le \bzero )
  \frac%
    {\prod_{j=1}^d \sigma_j \, e^{\gamma_j z_j}}%
    { \expec[ \max \{ (\bgamma \bR / \bsigma)^{1/\bgamma} \} ] }
  \int_0^\infty
    f_{\bR} \bigl( (\bsigma/\bgamma) \, (t \, e^{\bz})^{\bgamma} \bigr) \, 
    t^{\sum_{j=1}^d \gamma_j} \,
  \diff t.  
\end{align*}
Let $\bx$ be such that $\bx \not\le \bzero$ and $\sigma_j + \gamma_j x_j > 0$ for all $j = 1,\ldots,d$. Let $\bz = \bgamma^{-1} \log( 1 + \bgamma \bx / \bsigma )$. Then $\sigma_j \, e^{\gamma_j z_j} = \sigma_j + \gamma_j x_j$ and $(\sigma_j / \gamma_j) (t \, e^{z_j})^{\gamma_j} = ( \sigma_j + \gamma_j x) \, t^{\gamma_j} / \gamma_j$. By equation~\eqref{eq:Z2X:pdf}, we obtain
\begin{align*}
  h( \bx )
  &=
  h_{\bZ} \left( \bgamma^{-1} \log( 1 + \bgamma \bx / \bsigma ) \right) \,
  \prod_{j=1}^d \frac{1}{\sigma_j + \gamma_j x_j} \\
  &=
  \frac{1}{ \expec[ \max \{ (\bgamma \bR / \bsigma)^{1/\bgamma} \} ] }
  \int_0^\infty
    f_{\bR} \bigl( t^{\bgamma} ( \bx + \bsigma / \bgamma ) \bigr) \,
    t^{\sum_{j=1}^d \gamma_j} \,
  \diff t. \qedhere
\end{align*}
\end{proof}

\section*{Acknowledgments}

H. Rootz\'en's research was supported by the Knut and Alice Wallenberg foundation. J. Segers gratefully acknowledges funding by contract ``Projet d'Act\-ions de Re\-cher\-che Concert\'ees'' No.\ 12/17-045 of the ``Communaut\'e fran\c{c}aise de Belgique'' and by IAP research network Grant P7/06 of the Belgian government (Belgian Science Policy).


\bibliographystyle{apalike} 
\bibliography{libraryMGPD}

\begin{thebibliography}{}

\bibitem[Basrak and Segers, 2009]{basrak+s:2009}
Basrak, B. and Segers, J. (2009).
\newblock Regularly varying multivariate time series.
\newblock {\em Stochastic Process. Appl.}, 119(4):1055--1080.

\bibitem[Beirlant et~al., 2004]{beirlant2004}
Beirlant, J., Goegebeur, Y., Segers, J., and Teugels, J. (2004).
\newblock {\em Statistics of Extremes: Theory and Applications}.
\newblock John Wiley \& Sons.

\bibitem[de~Haan and Ferreira, 2006]{dehaanferreira2006}
de~Haan, L. and Ferreira, A. (2006).
\newblock {\em Extreme Value Theory: An Introduction}.
\newblock Springer.

\bibitem[Drees and Huang, 1998]{drees+h:98}
Drees, H. and Huang, X. (1998).
\newblock Best attainable rates of convergence for estimates of the stable tail
  dependence function.
\newblock {\em Journal of Multivariate Analysis}, 64:25--47.

\bibitem[Einmahl et~al., 2006]{einmahl+d+l:2006}
Einmahl, J. H.~J., de~Haan, L., and Li, D. (2006).
\newblock Weighted approximations of tail copula processes with application to
  testing the bivariate extreme value condition.
\newblock {\em The Annals of Statistics}, 34(4):1987--2014.

\bibitem[Einmahl et~al., 2012]{einmahl2012}
Einmahl, J. H.~J., Krajina, A., and Segers, J. (2012).
\newblock {An {M}-estimator for tail dependence in arbitrary dimensions}.
\newblock {\em The Annals of Statistics}, 40(3):1764--1793.

\bibitem[Falk et~al., 2010]{falk2010}
Falk, M., H{\"u}sler, J., and Reiss, R.-D. (2010).
\newblock {\em Laws of small numbers: extremes and rare events}.
\newblock Springer Science \& Business Media.

\bibitem[Ferreira and de~Haan, 2014]{ferreira2014}
Ferreira, A. and de~Haan, L. (2014).
\newblock The generalized {Pareto} process; with a view towards application and
  simulation.
\newblock {\em Bernoulli}, 20(4):1717--1737.

\bibitem[Joe et~al., 2010]{joe+l+n:2010}
Joe, H., Li, H., and Nikoloulopoulos, A.~K. (2010).
\newblock Tail dependence functions and vine copulas.
\newblock {\em Journal of Multivariate Analysis}, 101:252--270.

\bibitem[Kiriliouk et~al., 2016]{kiriliouk+rootzen+segers+wadsworth:2016}
Kiriliouk, A., Rootzén, H., Segers, J., and Wadsworth, J. (2016).
\newblock Peaks over thresholds modelling with multivariate generalized pareto
  distributions.
\newblock arXiv:1612.01773.

\bibitem[Marshall and Olkin, 1983]{marshall1983}
Marshall, A.~W. and Olkin, I. (1983).
\newblock Domains of attraction of multivariate extreme value distributions.
\newblock {\em The Annals of Probability}, 11(1):168--177.

\bibitem[Nelsen, 2006]{Nelsen06}
Nelsen, R.~B. (2006).
\newblock {\em An Introduction to Copulas}, volume 139 of {\em Lecture Notes in
  Statistics}.
\newblock Springer Verlag.

\bibitem[Pickands, 1981]{pickands81}
Pickands, III, J. (1981).
\newblock Multivariate extreme value distributions.
\newblock In {\em Proceedings of the 43rd session of the {I}nternational
  {S}tatistical {I}nstitute, {V}ol.\ 2 ({B}uenos {A}ires, 1981)}, volume~49,
  pages 859--878, 894--902.
\newblock With a discussion.

\bibitem[Ressel, 2013]{ressel:2013}
Ressel, P. (2013).
\newblock Homogeneous distributions—and a spectral representation of
  classical mean values and stable tail dependence functions.
\newblock {\em Journal of Multivariate Analysis}, 117:246--256.

\bibitem[Rootz\'en et~al., 2017]{rootzen+s+w:2017}
Rootz\'en, H., Segers, J., and Wadsworth, J.~L. (2017).
\newblock Multivariate peaks over thresholds models.
\newblock {\em Extremes}.
\newblock Forthcoming.

\bibitem[Rootz{\'e}n and Tajvidi, 2006]{rootzen2006}
Rootz{\'e}n, H. and Tajvidi, N. (2006).
\newblock Multivariate generalized {Pareto} distributions.
\newblock {\em Bernoulli}, 12(5):917--930.

\bibitem[Schlather and Tawn, 2002]{schlather+t:2002}
Schlather, M. and Tawn, J. (2002).
\newblock Inequalities for the extremal coefficients of multivariate extreme
  value distributions.
\newblock {\em Extremes}, 5(1):87--102.

\bibitem[Schmid et~al., 2010]{schmid+etal:2010}
Schmid, F., Schmidt, R., Blumentritt, T., Gaisser, S., and Ruppert, M. (2010).
\newblock Copula-based measures of multivariate association.
\newblock In Jaworski, P., Durante, F., H\"{a}rdle, W.~K., and Rychlik, T.,
  editors, {\em Copula Theory and Its Applications}, Lecture Notes in
  Statistics, pages 209--236. Berlin Heidelberg.

\bibitem[Schmidt and Stadtm\"uller, 2006]{schmidt+s:2006}
Schmidt, R. and Stadtm\"uller, U. (2006).
\newblock Non-parametric estimation of tail dependence.
\newblock {\em Scandinavian Journal of Statistics}, 33:307--335.

\bibitem[Segers, 2012]{segers2012}
Segers, J. (2012).
\newblock Max-stable models for multivariate extremes.
\newblock {\em REVSTAT -- Statistical Journal}, 10(1):61--82.

\bibitem[Sklar, 1959]{Sklar59}
Sklar, M. (1959).
\newblock Fonctions de r\'epartition \`a $n$ dimensions et leurs marges.
\newblock {\em Publ.\ Inst.\ Statist.\ Univ.\ Paris}, 8:229\textendash231.

\end{thebibliography}

\end{document}